\documentclass[a4paper,10pt]{article}
\usepackage[utf8]{inputenc}
\usepackage{latexsym}
\usepackage{amsmath}
\usepackage{amssymb}
\usepackage{amsthm}
\usepackage{amsfonts}
\usepackage[mathcal]{eucal}
\usepackage{color}
\usepackage{graphicx}
\usepackage{enumerate}
\usepackage[all,cmtip]{xy}
\usepackage{url}
\usepackage{placeins}
\usepackage[sc]{mathpazo}
\usepackage{mathabx}
\usepackage{titling}

\DeclareFontFamily{OT1}{pzc}{}
\DeclareFontShape{OT1}{pzc}{m}{it}{<-> s * [1.10] pzcmi7t}{}
\DeclareMathAlphabet{\mathpzc}{OT1}{pzc}{m}{it}

\newcommand{\N}{\mathbf{N}}

\newcommand{\Z}{\mathbf{Z}}

\newcommand{\kar}{\mathrm{char}}

\newcommand{\ck}{\mathpzc{k}}
\newcommand{\cV}{\mathpzc{V}}
\newcommand{\resp}{resp., }
\newcommand{\ie}{\textit{i.e.}, }
\newcommand{\dig}{\mathrm{dig}}
\newcommand{\clm}{\mathrm{clm}}
\newcommand{\minclm}{\mathrm{minclm}}
\newcommand{\pointref}[1]{(\ref{#1})}

\newcommand{\noqedhere}{\renewcommand{\qedsymbol}{}}
\newtheorem{theorem}{Theorem}
\newtheorem{lemma}{Lemma}
\newtheorem{conjecture}{Conjecture}
\newtheorem{definition}{Definition}
\newtheorem{proposition}{Proposition}
\newtheorem{example}{Example}
\newtheorem{remark}{Remark}

\title{Criteria for regularity of Mahler power series and Becker's conjecture}
\author{Tomasz Kisielewski}

\date{}

\begin{document}
\maketitle
 \begin{abstract}
		Allouche and Shallit introduced the notion of a regular power series as a
		generalization of automatic sequences. Becker showed that all
		regular power series satisfy Mahler equations and conjectured
		equivalent conditions for the converse to be true. We prove a stronger form of
		Becker's conjecture for a subclass of Mahler power series.
	\end{abstract}
	\setcounter{page}{2}
	\section{Introduction}
	 We first give some historical background and motivation for the problems
		in this paper, starting with automatic and regular sequences and moving
		on to power series and Mahler equations.
		\subsection{Automatic and regular sequences}
			Fix a natural number $k \in \N$, $k \geq 2$. A sequence $\left( a_n \right)$
			is called $k$-automatic if there exists a
			finite automaton which, given the base $k$ expansion of $n$, stops at a state
			corresponding to $a_n$. This definition of $k$-automatic sequences was
			introduced by Cobham \cite{Cobham72}, who also proved an alternative
			characterisation of these sequences, which will be much more useful for us.
			\begin{definition}
				Let $\left( a_i \right)_{i \in \N}$ be a sequence. The set of subsequences
				of this sequence $\left\{ \left( a_{k^e i + r} \right)_{i \in \N} \colon e
				\in \N, 0 \leq r < k^e \right\}$ is called the \emph{$k$-kernel} of the
				sequence.
			\end{definition}
			\begin{definition}
				A sequence $\left( a_i \right)_{i \in \N}$ is called \emph{$k$-automatic} if
				its $k$-kernel is finite.
			\end{definition}
			Based on this definition, Allouche and Shallit \cite{AlloucheShallit92}
			defined and investigated a wider class of sequences. They were working with
			sequences taking values in rings, but for simplicity we will restrict our
			attention to a field $\ck$.
			\begin{definition}
				A sequence $\left( a_i \right)_{i \in \N}$, $a_i \in \ck$, is called
				\emph{$k$-regular} if the vector subspace of $\ck^\N$ generated over $\ck$ by
				its $k$-kernel is finitely dimensional.
			\end{definition}
			Among many other interesting properties, Allouche and Shallit state the following theorem.
			\begin{theorem}\label{thm:regularityAutomaticity}
				A sequence is $k$-automatic if and only if it is $k$-regular and takes
				only finitely many values.
			\end{theorem}
			This theorem fully describes the relation between regular and automatic
			sequences.

			Automatic and regular sequences have many applications in, among others,
			transcendence theory, game theory, questions related to computability
			properties of expansions of numbers in various bases, dynamical systems,
			differential geometry, Fourier analysis and language theory. We refer the interested
			reader to the excellent book by Allouche and Shallit
			\cite{AlloucheShallit03}. For examples of automatic and regular sequences we
			refer to the same book and also the papers \cite{AlloucheShallit92} and
			\cite{AlloucheShallit03paper} by the same authors.
		\subsection{Power series and Mahler equations}
		 First we note that the definitions from the previous section can be
			naturally extended to formal power series.
			\begin{definition}
				A formal power series $\sum_{i = 0}^{\infty} a_i z^i = f(z) \in \ck\left[
				\left[	z \right] \right]$ is called $k$-regular (\resp $k$-automatic) if the
				sequence of its coefficients $\left( a_i \right)_{i \in \N}$ is $k$-regular
				(\resp $k$-automatic).
			\end{definition}
			This definition was also established by Allouche and
			Shallit \cite{AlloucheShallit92} along with some basic results, among which
			was the fact that $k$-regular power series form a ring, but not a field.

			In a series of papers concerning transcendence theory Mahler \cite{Mahler29,
			Mahler30-1, Mahler30-2} considered a new type of functional equations.
			He proved that any power series with coefficients in a number field
			satisfying an equation of such a type takes a
			transcendental value at algebraic points in the radius of convergence. His
			method was later extended by many others and we refer the interested
			reader to Nishioka \cite{Nishioka96}.
			\begin{definition}
				A formal power series $f(z) \in \ck\left[ \left[ z \right] \right]$ is
				called \emph{$k$-Mahler} if it satisfies a functional equation of the form
				\begin{equation*}
					f\left( z \right) = \sum_{i = 1}^{n} c_i\left( z \right) f\left(
					z^{k^i} \right)
				\end{equation*}
				for some rational functions $c_i\left( z \right) \in \ck\left( z
				\right)$. We call $n$ the order of the Mahler equation.
			\end{definition}
			There are a few equivalent definitions of $k$-Mahler power series which
			differ slightly in the precise form of the satisfied equation. For these definitions and
			proofs of equivalence we refer the reader to Adamczewski and
			Bell \cite{AdamczewskiBell13}, Becker \cite{Becker92} and Dumas \cite{Dumas93}.

			The latter two papers also explore the relation between $k$-regular and
			$k$-Mahler power series. These two classes of power series are rather close.
			More precisely, all $k$-regular series are $k$-Mahler and the
			strongest known result considering the converse is the following theorem.
			\begin{theorem}\label{thm:DumasReguralityCriterion}
				Let the power series $f\left( z \right) \in \ck\left[ \left[ z
				\right] \right]$ satisfy a Mahler type equation with coefficients $c_i\left(
			 z \right) \in \ck\left( z \right)$. If all $c_i\left( z \right)$ have poles
			 only at $0$ and roots of unity of order not coprime to $k$, then $f\left( z
			 \right)$ is $k$-regular.
			\end{theorem}
			This theorem has been shown by Dumas \cite[Theorem 30]{Dumas93}, while
			Becker \cite[Theorem 2]{Becker92} slightly earlier showed a weaker result, with
			$c_i\left( z \right)$ being polynomials. Due to that, Adamczewski and
			Bell used the name \emph{$k$-Becker} for power series satisfying a
			Mahler equation with polynomial coefficients, so by analogy we will use the
			name \emph{$k$-Dumas} for power series satisfying the conditions of Theorem
			\ref{thm:DumasReguralityCriterion}. One of our main results is that with
			some additional assumptions the class of $k$-Dumas power series is exactly
			the class of $k$-regular power series (Theorem \ref{thm:fullEquivalence}) and
			we conjecture that this is also the case in general (Conjecture
			\ref{conj:fullEquivalence}).

			Becker also used the fact that $k$-regular power series form a ring to
			conjecture that all $k$-regular power series are  quotients of $k$-Becker
			power series and polynomials. We investigate this claim and show that it
			would follow from the converse of Theorem
			\ref{thm:DumasReguralityCriterion}.

			The relations between regular, Becker, and Mahler power series are not only
			interesting as a matter of mathematical curiosity. They appeared in the proof
			of an extended version of Cobham's theorem by Adamczewski and Bell
			\cite{AdamczewskiBell13}. The original Cobham's theorem states that a sequence
			that is $k$- and $l$-automatic for two multiplicatively independent integers
			$k$ and $l$ (that is $k$ is not a power of $l$ and $l$ is not a power of $k$)
			is ultimately periodic. Adamczewski and Bell first extended this to the
			case of regular sequences and then proved that $k$- and $l$-Mahler power
			series for multiplicatively independent integers $k$ and $l$ are rational
			functions. The latter proof was much more complicated than the former
			(using methods from commutative algebra and Chebotar\"ev's theorem as well as
			analysis), so a simpler characterisation of the relation between regular and
			Mahler power series might have helped. They used the characterisation given
			by the following lemma, which is a slightly modified version of an original
			result of Dumas \cite[Theorem 31]{Dumas93}.
			\begin{lemma}\label{lem:MahlerAlmostRegular}
				Let $f\left( z \right) \in \ck\left[ \left[ z \right] \right]$ be a Mahler
				power series satisfying the equation
				\begin{equation*}
					\sum_{i = 0}^{n} c_i\left( z \right) f\left( z^{k^i} \right) = 0
				\end{equation*}
				where $c_i\left( z \right) \in \ck\left[ z \right]$ are polynomials and
				$c_0\left( 0 \right) = 1$. Then there exists a Becker power series $g\left( z
				\right)$ such that
				\begin{equation*}
					f\left( z \right) = \left( \prod_{i = 0}^\infty c_0\left( z^{k^i} \right)
					\right)^{-1} g\left( z \right).
				\end{equation*}
			\end{lemma}
			Note that $h\left( z \right) = \left( \prod_{i = 0}^\infty c_0\left( z^{k^i}
			\right)\right)^{-1}$ is Mahler since it satisfies the Mahler equation
			\begin{equation*}
				h\left( z \right) = \frac{1}{c_0\left( z \right)} h\left( z^k \right).
			\end{equation*}

			This again shows that regular and Mahler power series are quite close.
		\subsection{Plan of this work}
			While the results of sections \ref{subs:valuations} and
			\ref{subs:mahlerOperators} are well-known, most of the other definitions and
			results are new. The methods follow essentially Dumas \cite[Proposition
			54]{Dumas93}, but using the language of valuations and having some patience
			for computations enables us to extend the results to more general situations.

			In section \ref{sect:definitions} we introduce the basic notions used in this paper.
			Section \ref{sect:orderOne} describes the special case of Mahler equations of order one, in particular
			a full characterisation of their regularity. Section \ref{sect:calmnesses} contains one of our main
			results (Theorem \ref{thm:precalmness}) which gives a numerical criterion for precalmness, a property of
			coefficents in Mahler equations related to regularity. Sections \ref{sect:conjectures} and \ref{sect:naive}
			explore various forms of Becker's conjecture. The former shows the equivalence of some forms, while the latter states
			a naive version and constructs a counterexample, which also illustrates how to apply Proposition \ref{prop:technicalFullEquivalence}
			from a computational point of view. In section \ref{sect:criteria} we state and prove equivalent conditions
			for regularity for a subclass of Mahler power series.
		\subsection*{Acknowledgements}
			This work was originally a master's thesis written under the supervision of dr. Jakub Byszewski,
			to whom I would like to express my gratitude for all the encouragement and mathematical support.
	\section{Preliminary definitions}\label{sect:definitions}
	 This section contains most of the definitions and some basic lemmas used in
		later proofs. We will work with a fixed algebraically closed field $\ck$ of
		characteristic $\kar\ck = p$ or $0$ and with a fixed integer $k \geq 2$.
		\emph{If $\kar\ck = p > 0$ we always assume that $k$ is coprime to $p$.}
		\subsection{Valuations in the field of rational functions}\label{subs:valuations}
		 We will be using the language of valuations to talk about poles and zeroes of
			functions. We will denote by $v_{\alpha}$ the $\left( z - \alpha	\right)$-adic valuation on $\ck\left( z \right)$.

			We will be using mostly basic properties of valuations, but we will actually need a more precise tool for determining the valuation of a sum of
			two rational functions, involving their $\left( z - \alpha \right)$-adic digits. We will denote the $\left( z - \alpha \right)$-adic digit of a rational
			function $c\left( z \right) \in \ck\left( z \right)$ at $\left( z - \alpha \right)^i$ by $\dig_{\alpha, i}\left( c\left( z \right) \right)$.
			The valuation of such a rational function is equal to the lowest
			index for which its $\left( z - \alpha \right)$-adic digit is
			not zero. To get the
			information necessary to use the above fact in our
			proofs, we will need a simple lemma.
			\begin{lemma}\label{lem:arbitraryDigits}
				Let $c\left( z \right), d\left( z \right) \in \ck\left( z \right)$ be two
				rational functions, $\alpha_0, \dots, \alpha_n \in \ck$ be pairwise distinct
				numbers and $m_0, \dots, m_n \in \Z$ be integers, $m_i \geq
				v_{\alpha_i}\left( c\left( z \right) \right)$. If $v_{\alpha_i}\left(
				d\left( z \right) \right) \geq v_{\alpha_i}\left( c\left( z \right) \right)$
				for every $0 \leq i \leq n$, then there exists a polynomial $h\left( z
				\right) \in \ck\left[ z \right]$ such that $\dig_{\alpha_i, j}\left( h\left(
				z \right)c\left( z \right)\right) = \dig_{\alpha_i, j}\left( d\left( z
				\right)	\right)$ for all $0 \leq i \leq n$ and $v_{\alpha_i}\left( c\left(
				z \right)	\right) \leq j \leq m_i$.
			\end{lemma}
			This follows from the Approximation Lemma from \cite[p. 12]{Serre79} applied
			to the single function $\frac{d\left( z \right)}{c\left( z \right)}$, the
			prime ideals generated by $\left( z - \alpha_i \right)$ and integers $m_i$.
			In fact, it is enough to choose $h\left( z \right)$ so that
			\begin{equation*}
				v_{\alpha_i}\left( h\left( z \right) - \frac{d\left( z \right)}{c\left( z
				\right)} \right) \geq m_i - v_{\alpha_i}\left( c\left( z \right) \right).
			\end{equation*}

		 There is one other simple property of valuations we will use.
			\begin{lemma}\label{lem:valuationUnderMahler}
				Let $i \in \N$ be an integer coprime to $p = \kar\ck$. Assume $\alpha \neq
				0$. Then
				\begin{equation*}
					v_{\alpha}\left( c\left( z^i \right) \right) = v_{\alpha^i}\left( c\left(
					z \right) \right).
				\end{equation*}
				\begin{proof}
					For any $v \in \N$ the polynomial $\left( \frac{\left( \alpha^i -
					z^i \right)}{\left( \alpha - z \right)} \right)^{v}$ has no zeroes at $\alpha$, by the
					assumption of $i$ being coprime to the characteristic of $\ck$.
					Thus, the function $\left( \alpha - z
					\right)^{v} c\left( z^i \right)$ has no poles at $\alpha$ if and only if
					it has no poles at $\alpha$ after being multiplied by this polynomial.
					But this
					means that $d\left( z \right) = \left( \alpha^i - z^i \right)^{v} c\left(
					z^i \right)$ also has no poles at $\alpha$. Substituting the variable
					$t = z^i$ into $\left( \alpha^i - z \right)^{v} c\left( z \right)$, we get
					$d\left( t \right)$, so the former expression has no poles at $\alpha^i$ if and only
					if $d\left( t \right)$ has no poles at $\alpha$.
				\end{proof}
			\end{lemma}
		\subsection{Calm numbers, calm functions and calm sequences of functions}
			Due to the somewhat complex criteria for poles in Theorem
			\ref{thm:DumasReguralityCriterion}, we set up some nomenclature for numbers
			and functions satisfying these criteria.
			\begin{definition}
				A number $\alpha \in \ck$ is called \emph{$k$-calm} if $\alpha$ is a root of unity
				of order not coprime to $k$ or $\alpha = 0$. If a number is not $k$-calm, we
				call it \emph{$k$-anxious}.
			\end{definition}

			\begin{definition}
				Let $\alpha$ be a $k$-anxious number. A rational function $c\left( z \right)
				\in \ck\left( z \right)$ is called \emph{$\alpha$-calm} if it has no poles
				at $\alpha$.
			\end{definition}

			\begin{definition}\label{def:singleCalm}
				A rational function $c\left( z \right) \in \ck\left( z \right)$ is called
				\emph{$k$-calm} if it has poles only at $k$-calm numbers (\ie the function
				is $\alpha$-calm for all $k$-anxious $\alpha$).
			\end{definition}

			Since we will be working with sequences of rational functions rather than a
			single function, we extend these definitions as follows.
			\begin{definition}\label{def:sequenceCalm}
				Let $\alpha$ be a $k$-anxious number. A finite sequence of rational
				functions $\left( c_i\left( z \right) \right)_{i = 1}^n$, $c_i\left( z
				\right) \in \ck\left( z \right)$ is called
				$\alpha$-calm (\resp $k$-calm) if all the functions $c_i\left( z \right)$ are
				$\alpha$-calm (\resp $k$-calm).
			\end{definition}
			We will omit the $k$-
			prefixes in the notation whenever it is clear from context.
			In this language Theorem \ref{thm:DumasReguralityCriterion} can be stated as
			follows -- if a power series satisfies a Mahler equation with a calm
			sequence of coefficients, then it is regular.

			For the discussion of the Becker conjecture we will also need some further
			definitions.
			\begin{definition}
				We define the action of the group $\ck\left( z \right)^*$ on the set of
				sequences $\ck\left( z \right)^{n}$.
				Let $\left( c_i\left( z \right) \right)_{i = 1}^n$, $c_i\left( z
				\right) \in \ck\left( z \right)$ be a finite sequence of rational functions
				and $0 \neq h\left( z \right) \in \ck\left( z \right)$ be a nonzero rational function. We
				define the action of the rational function $h\left( z \right)$ on the
				sequence $\left( c_i\left( z \right) \right)_{i = 1}^n$ by the formula
				\begin{equation*}
					h^*\left( c_i\left( z \right) \right) = \left( \frac{h\left( z^{k^i} \right)}{h\left( z
					\right)} c_i\left( z \right) \right).
				\end{equation*}
			\end{definition}

			\begin{definition}
				Let $\alpha$ be a $k$-anxious number. A finite sequence of rational
				functions $\left( c_i(z) \right)_{i = 1}^{n}$, $c_i\left( z \right) \in
				\ck\left( z \right)$ is called \emph{$\alpha$-precalm} if there exists a
				polynomial $0 \neq h\left( z \right) \in \ck\left[ z \right]$ such that $h^*\left( c_i\left(
				z \right) \right)$ is $\alpha$-calm.
			\end{definition}

			\begin{definition}
				A finite sequence of rational functions $\left( c_i(z) \right)_{i =
				1}^{n}$, $c_i\left( z \right) \in \ck\left( z \right)$ is called
				\emph{$k$-precalm} if there exists a polynomial $0 \neq h\left( z
				\right) \in \ck\left[ z
				\right]$ such that $h^*\left( c_i\left( z \right) \right)$ is $k$-calm.
			\end{definition}
			It will be obvious from the proof of Theorem \ref{thm:precalmness} that being
			$k$-precalm is equivalent to being $\alpha$-precalm for all $k$-anxious
			$\alpha$, but this is not completely trivial, so we only prove this later.

			To prove conditions equivalent to precalmness we will also need the
			following quantitative measure of how far a function is from being calm.
			\begin{definition}\label{def:calmness}
				For a sequence $\left( a_i \right)$ we will write $\overline{a}_i$ for its
				associated sequence of partial sums $\overline{a}_i = \sum_{j = 0}^{i-1}
				a_j$, with $\overline{a}_0 = 0$. For any $k$-anxious number
				$\alpha$ put
				\begin{equation*}
					A_\alpha =
					\begin{cases}
						\left\{ (a_i)_{i = 0}^{\infty} \colon a_i \in \left\{ 1,\dots,n \right\} \right\}
						& \textrm{ if } \alpha \textrm{ is not a root of unity,}\\
						\left\{ (a_i)_{i = 0}^{\mu-1} \colon a_i \in \left\{ 1,\dots,n \right\},
						\alpha^{k^{\overline{a}_\mu}} = \alpha \right\}
						& \textrm{ if } \alpha \textrm{ is a root of unity.}
					\end{cases}
				\end{equation*}
				The \emph{$k$-calmness} of a finite sequence of rational functions $\left(
				c_i(z) \right)_{i = 1}^{n}$, $c_i\left( z \right) \in \ck\left( z \right)$
				with respect to a $k$-anxious number $\alpha$ and a sequence
				$(a_i) \in A_\alpha$ is
				\begin{equation*}
					\clm_{\alpha, (a_i)} \left( c_i \right) = \sum_i
					v_{\alpha^{k^{\overline{a}_i}}}\left( c_{a_i}\left( z \right) \right).
				\end{equation*}
				The sum is taken over the set $\left\{ 0, \dots, \mu - 1 \right\}$ for
				$\alpha$ a root of unity and over all natural numbers if it is not a root of
				unity. For simplicity we will usually omit $\alpha$ from the notation.

				It is worth noting that this sum is always either finite or contains an
				infinite term. For a root of unity $\alpha$ this is obvious since the sum has
				finitely many terms. If $\alpha$ is not a root of unity then any nonzero
				rational function has only finitely many poles and zeroes, so the sum has
				only finitely many nonzero terms or contains at least one infinite term
				(which can happen only when some $c_{a_i}\left( z \right)$ is zero).
			\end{definition}
			We will also use the connection between calmness and the action of rational
			functions on the set of sequences.
			\begin{lemma}\label{lem:actionEffect}
				Let $\alpha$ be an anxious number, $(a_i)\in A_\alpha$ and $0 \neq h\left( z
				\right) \in \ck\left( z \right)$ be a nonzero rational function. Then
				\begin{equation*}
					\clm_{(a_i)} h^*(c_i) =
					\begin{cases}
						\clm_{\left( a_i \right)} (c_i) - v_\alpha\left( h\left( z \right) \right)
						& \textrm{ if } \alpha \textrm{ is not a root of unity,}\\
						\clm_{\left( a_i \right)} (c_i)
						& \textrm{ if } \alpha \textrm{ is a root of unity.}
					\end{cases}
				\end{equation*}
				\begin{proof}
					\begin{multline*}
						\clm_{(a_i)} h^*(c_i) = \sum_i v_{\alpha^{k^{\overline{a}_i}}}\left(
							\frac{h\left( z^{k^{a_i}} \right)}{h\left( z \right)} c_{a_i}\left( z
							\right) \right) \\
						\begin{split}
							&= \sum_i \left( v_{\alpha^{k^{\overline{a}_i}}}\left( c_{a_i}\left(
							z \right) \right) + v_{\alpha^{k^{\overline{a}_i}}}\left( h\left( z^{k^{a_i}}
							\right) \right) - v_{\alpha^{k^{\overline{a}_i}}}\left( h\left( z
							\right) \right) \right) \\
							&= \sum_i v_{\alpha^{k^{\overline{a}_i}}}\left(
							c_{a_i}\left( z \right) \right) + \sum_i \left(
							v_{\alpha^{k^{\overline{a}_{i+1}}}}\left(	h\left( z \right) \right) -
							v_{\alpha^{k^{\overline{a}_i}}}\left( h\left( z	\right) \right) \right) \\ &=
							\begin{cases}
								\clm_{(a_i)} (c_i) - v_\alpha\left( h\left( z \right) \right)
								& \textrm{ if } \alpha \textrm{ is not a root of unity,}\\
								\clm_{(a_i)} (c_i) - v_\alpha\left( h\left( z \right) \right) +
								v_{\alpha^{k^{\overline{a}_\mu}}} \left( h\left( z \right) \right) =
								\clm_{(a_i)} (c_i)
								& \textrm{ if } \alpha \textrm{ is a root of unity.}
							\end{cases}
						\end{split}
					\end{multline*}
					In the second line, we used Lemma
					\ref{lem:valuationUnderMahler} and the definition of $\overline{a}_i$,
					while the last equality for roots of unity holds by the definition
					of $A_\alpha$.
				\end{proof}
			\end{lemma}
			We will prove more results connected to precalmness when discussing Becker's
			conjecture. For some of that discussion, we will need one more property of
			calm sequences of functions.
			\begin{lemma}\label{lem:prepolynomial}
				Let $\left( c_i\left( z \right) \right)_{i = 1}^n$ be a $k$-calm sequence of
				rational functions. Then there exists a polynomial $0 \neq h\left( z \right) \in
				\ck\left[ z \right]$ such that $h^*\left( c_i\left( z \right) \right)$ is a
				sequence of polynomials.
				\begin{proof}
					We will first construct polynomials $h_i'\left( z \right)$ for $1 \leq i \leq
					n$ such that
					\begin{enumerate}[(a)]
						\item\label{it:addNoPoles} $\frac{h_i'\left( z^{k^j} \right)}{h_i'\left(
							z \right)}$ is a polynomial for every $1 \leq j \leq n$.
						\item\label{it:polynomialize} $\frac{h_i'\left( z^{k^i} \right)}{h_i'\left(
							z \right)}c_i'\left( z \right)$ is a polynomial.
					\end{enumerate}
					It then remains to put $h\left( z \right) = \prod_{i = 1}^n h_i'\left(
					z \right)$. Property \pointref{it:addNoPoles} ensures that 
					$h\left( z \right)$ will not add any poles to the terms of the new sequence
					and property \pointref{it:polynomialize} shows that $h\left( z
					\right)$ will eliminate any poles that $c_i\left( z	\right)$ had before.

					Let us fix $1 \leq i \leq n$. We construct $h_i'\left( z \right)$. For
					any $k$-calm $\alpha$ consider the polynomial
					\begin{equation*}
						H_\alpha\left( z \right) = \begin{cases}
							z & \textrm{ if } \alpha = 0,\\
							z^{\frac{b}{q}} - 1 &
							\textrm{ if } \alpha = \zeta_b \textrm{ is a } b^{\mathrm{th}} \textrm{
							primitive root of unity with } q = \gcd\left( k, b \right) \neq 0.
						\end{cases}
					\end{equation*}
					Set
					\begin{equation*}
						h_i'\left( z \right) = \prod_{\alpha\colon v_\alpha\left( c_i' \right) < 0} H_\alpha\left( z
						\right)^{-v_\alpha\left( c_i' \right)}.
					\end{equation*}
					It remains to prove that this function has the desired properties. Property
					\pointref{it:addNoPoles} can be checked for every $H_\alpha$ separately, since if
					two functions satisfy property \pointref{it:addNoPoles}, so does their
					product. For $\alpha = 0$ the property is obvious since $\frac{H_0\left(
						z^{k^j} \right)}{H_0\left( z \right)} = z^{k^j - 1} \in \ck\left[ z
						\right]$. For $\alpha = \zeta_b$ we have
					\begin{equation*}
						\frac{H_\alpha\left( z^{k^j} \right)}{H_\alpha\left( z \right)} =
						\frac{z^{k^j \frac{b}{q}} - 1}{z^{\frac{b}{q}} - 1} = \frac{\left(
							z^{\frac{b}{q}} \right)^{k^j} - 1}{\left( z^{\frac{b}{q}} \right) - 1} =
							\sum_{c = 0}^{k^j - 1} \left( z^{\frac{b}{q}} \right)^c,
					\end{equation*}
					which is indeed a polynomial.

					To prove that property \pointref{it:polynomialize} is satisfied, it suffices to
					show that every pole of $c_i'$ counted with multiplicity corresponds to a
					root of $\frac{h_i'\left( z^{k^i} \right)}{h_i'\left( z
					\right)}$ of at least the same multiplicity. For every such pole $\alpha$,
				 the polynomial $h_i$ has a corresponding factor $H_\alpha$ and so it
					suffices to show that $\frac{H_\alpha\left( z^{k^i} \right)}{H_\alpha\left(
					z \right)}$ has a root at $\alpha$. For $\alpha = 0$ this is obvious. For
					$\alpha = \zeta_b$ we have
					\begin{equation*}
						\left( \zeta_b \right)^{k^i\frac{b}{q}} - 1 = \left( \zeta_b^b
						\right)^{\frac{k^i}{q}} - 1 = 0.
					\end{equation*}
					However, $\zeta_b$ is not a root of $z^{\frac{b}{a}} - 1$, so
					$\frac{H_\alpha\left( z^{k^i} \right)}{H_\alpha\left( z \right)} =
					\frac{z^{k^j \frac{b}{q}} - 1}{z^{\frac{b}{q}} - 1}$ has a root at
					$\zeta_b$.
				\end{proof}
			\end{lemma}
		\subsection{Ring of Mahler operators}\label{subs:mahlerOperators}
		 We will also use a slightly different approach to Mahler equations. It was
			used extensively by Dumas \cite{Dumas93} in his thesis. For more
			information about the notions defined in this section, we direct the reader
			there.
			\begin{definition}
				We define the operator $\Delta_k \colon \ck\left( \left( z \right) \right)
				\to \ck\left( \left( z \right) \right)$ as
				\begin{equation*}
					\Delta_k f\left( z \right) = f\left( z^k \right)
				\end{equation*}
				for any formal Laurent series $f\left( z \right) \in \ck\left( \left( z
				\right) \right)$.
			\end{definition}

			We regard rational functions $c\left( z \right) \in \ck\left( z
			\right)$ as multiplication by $c\left( z \right)$ operators on $\ck\left( \left( z \right) \right)$.
			\begin{definition}
				We call the (non-commutative) ring $\ck\left( z \right)\left[ \Delta_k
				\right]$
				\emph{the ring of $k$-Mahler operators}. The multiplication in this ring
				corresponds to composition of operators and the multiplication by rational functions together with the rule
				\begin{equation}\label{eq:MahlerOperatorRule}
					\Delta_k c\left( z \right) = c\left( z^k \right)\Delta_k
				\end{equation}
				for any rational function $c\left( z \right) \in \ck\left( z \right)$.
				We regard the ring $\ck\left( z \right)\left[ \Delta_k \right]$ as a subring
				of the ring of $\ck$-linear maps $\ck\left( \left( z \right) \right) \to \ck\left(
				\left( z \right) \right)$ that are continuous in the $z$-adic topology.
			\end{definition}
			\begin{remark}\label{rem:MahlerIsFreeModule}
				The ring $\ck\left( z \right)\left[ \Delta_k \right]$ is a free left
				$\ck\left( z \right)$-module with basis $1, \Delta_k, \Delta_k^2, \dots$.
				\begin{proof}
					Obvious from equation \eqref{eq:MahlerOperatorRule}.
				\end{proof}
			\end{remark}

			This definition allows us to reformulate the definition of a $k$-Mahler power
			series.
			\begin{remark}\label{rem:MahlerOperatorMahlerEquation}
				A formal power series $f\left( z \right) \in \ck\left[ \left[ z
				\right] \right]$ is called $k$-Mahler if there exists an operator $M \in 1 +
				\ck\left( z \right)\left[ \Delta_k \right]\Delta_k$ such that
				\begin{equation*}
					M f(z) = 0.
				\end{equation*}
			\end{remark}
			From the above it is obvious that multiplying $M$ on the left by any element $N \in  1 +
			\ck\left( z \right)\left[ \Delta_k \right]\Delta_k$ gives us a new
			$k$-Mahler equation satisfied by $f\left( z \right)$.

			We introduce a family of operators providing us with another way
			of looking at regular power series. These operators have been studied by
			Dumas \cite{Dumas93} and Becker \cite{Becker92}.
			\begin{definition}
				We define the operator $\Lambda_{k, r} \colon \ck\left[ \left[ z
				\right] \right] \to \ck\left[ \left[ z \right] \right]$ for $0 \leq r < k$
				by the formula
				\begin{equation*}
					\Lambda_{k, r} \left( \sum_{i = 0}^{\infty} a_i z^i \right) = \sum_{i = 0}^{\infty} a_{ki + r} z^i.
				\end{equation*}
				These operators are sometimes called the Cartier operators.
			\end{definition}
			It is quite easy to see that iterated application of the operators
			$\Lambda_{k, r}$ on a power series $f\left( z \right) \in \ck\left[ \left[
			z \right] \right]$ gives us power series corresponding to sequences in the
			$k$-kernel of the sequence associated to $f\left( z \right)$. This proves the
			following lemma, which was first noted by Becker \cite[Lemma 3]{Becker92}.
			\begin{lemma}\label{lem:lambdaRegularity}
				A power series $f\left( z \right) \in \ck\left[ \left[ z \right] \right]$ is
				$k$-regular if and only if the $\ck$-vector space generated by the set
				\begin{equation*}
					\left\{ \Lambda_{k, r_n}\left( \Lambda_{k, r_{n - 1}} \left( \dots \left( \Lambda_{k, r_0}\left( f(z)
					\right) \right) \right) \right) \colon 0 \leq r_i < k, n \in \N \right\}
				\end{equation*}
				is finitely dimensional.
			\end{lemma}
			We will require a few more simple properties of the Cartier operators.
			\begin{lemma}\label{lem:lambdaProp}
				Let $f(z), g(z) \in \ck\left[ \left[ z \right] \right]$ be power series.
				Then
				\begin{enumerate}[(a)]
					\item\label{lambdaForm}
						\begin{equation*}
							f(z) = \sum_{r = 0}^{k - 1} z^r \Lambda_{k, r}\left( f \right)\left( z^k	\right)
						\end{equation*} and
					\item\label{lambdaProd}
						\begin{equation*}
							\Lambda_{k, r}\left( f(z) g\left( z^k \right) \right) = \Lambda_{k, r}\left( f(z)
							\right) g(z).
						\end{equation*}
				\end{enumerate}
				\begin{proof}
					The operator $\Lambda_{k, r}$ is $\ck$-linear and continuous in the
					$\left( z - \alpha \right)$-adic topology, so it
					suffices to prove the lemma for $f\left( z \right) = z^n$ and $g\left( z
					\right) = z^m$. For part \pointref{lambdaForm} write $n = kt + r_0$, $0 \leq
					r_0 < k$ and note that
					\begin{equation*}
						\Lambda_{k, r}\left( z^n \right) = 
						\begin{cases}
							z^t & r = r_0\\
							0 & r \neq r_0.
						\end{cases}
					\end{equation*}
					The proof reduces to the computation
					\begin{equation*}
						\sum_{r = 0}^{k - 1} z^r \Lambda_{k, r}\left( f \right)\left( z^k	\right)
						= z^{r_0} \left( z^k \right)^t = z^n
					\end{equation*}
					For part \pointref{lambdaProd} note that
					\begin{equation*}
							\Lambda_{k, r_0}\left( f(z) g\left( z^k \right) \right) = \Lambda_{k,
							r_0}\left( z^{n + km} \right) = z^{t + m} = z^t z^m = \Lambda_{k,
							r_0}\left( f\left( z \right) \right) g\left( z \right)
					\end{equation*}
					and
					\begin{equation*}
							\Lambda_{k, r}\left( f(z) g\left( z^k \right) \right) = \Lambda_{k,
							r}\left( z^{n + km} \right) = 0 = \Lambda_{k, r}\left( f\left( z \right)
							\right) g\left( z \right)
					\end{equation*}
					for $r \neq r_0$.
				\end{proof}
			\end{lemma}
			The following lemma will be crucial in proving the non-regularity of certain Mahler
			power series.
			\begin{lemma}\label{lem:polePerserving}
				Let $c\left( z \right) \in \ck\left( z \right)$ be a rational function,
				$0 \neq \alpha \in \ck$ be a nonzero number and let $v = v_{\alpha}\left(
				c\left( z \right) \right)$. Then there exists an $r \in \left\{ 0,
				\dots, k - 1	\right\}$, such that $v_{\alpha}\left(\Lambda_{k, r}\left( c
				\right)\left( z^k \right) \right) \leq v$.
				\begin{proof}
					First we note that multiplying by $z^i$ has no effect on the $\left( z -
					\alpha \right)$-adic valuation of a rational function. The lemma then follows directly from Lemma
					\ref{lem:lambdaProp}.\pointref{lambdaForm}.
				\end{proof}
			\end{lemma}

			As before, we will write $\Lambda_r$ for $\Lambda_{k, r}$ whenever there is
			no risk of confusion.
	\section{Order one Mahler equations}\label{sect:orderOne}
		In order to show the basic principle behind our methods, we start with the case
		of a formal power series $f\left( z \right) \in \ck\left[ \left[ z
		\right] \right]$ satisfying a Mahler equation of order one, that is
		\begin{equation}\label{eq:orderOne}
			f\left( z \right) = c_1\left( z \right)f\left( z^k \right)
		\end{equation}
		for some rational function $c_1\left( z \right) \in \ck\left( z
		\right)$. The results of this section will not be used later, but
		they should help the reader understand the origin of the methods used in the
		following sections. Equations of order one are quite special and the
		criterion for regularity we provide will be a bit simpler than the criteria
		proven later for arbitrary order equations. Furthermore, we do not restrict
		the equations under consideration except for the order, which is not the case
		in later theorems.
		\subsection{Calmness of a single function}
			We will first show how the notion of calmness defined in Definition \ref{def:calmness}
			corresponds to functions being precalm. For simplicity we will say that a
			single function is (pre)calm if the sequence of length one containing only
			this function is (pre)calm. Note that for sequences of length one
			Definition \ref{def:sequenceCalm} reduces to Definition \ref{def:singleCalm},
			so saying that a function is calm is unambiguous. Furthermore, such sequences
			have only one nontrivial calmness for any $\alpha$, which we will henceforth
			denote by $\clm_\alpha = \clm_{\alpha, \left( 1 \right)}$.
			\begin{lemma}\label{lem:precalm}
				Given a rational function $c\left( z \right) \in \ck\left( z \right)$, the
				following conditions are equivalent:
				\begin{enumerate}[(i)]
					\item\label{lem:it:calmness} $\clm_\alpha (c) \geq 0$ for every anxious $\alpha$.
					\item $c\left( z \right)$ is precalm.
				\end{enumerate}
			\end{lemma}
			We will later prove Theorem \ref{thm:precalmness}, which is an exact analogue
			of this lemma for sequences of functions.

			\begin{proof}
				\begin{description}
					\item[(ii) $\Rightarrow$ (i)]
						In the language of valuations we can express Definition
						\ref{def:singleCalm} as follows: a function is calm if and only if its valuation
						$v_\alpha$ is nonnegative for every anxious $\alpha$. This implies that
						the calmness of a calm function is nonnegative.

						On the other hand, since $c\left( z \right)$ is precalm, we get a
						polynomial $0 \neq h\left( z \right) \in \ck\left[ z \right]$, such that
						$h^*\left( c\left( z \right) \right) = \frac{h\left( z^k \right)}{h\left(
						z \right)} c\left( z \right)$ is calm. By Lemma \ref{lem:actionEffect},
						\begin{equation*}
							0 \leq
							\begin{cases}
								\clm_\alpha (c) - v_\alpha\left( h\left( z \right) \right)
								& \textrm{ if } \alpha \textrm{ is not a root of unity,}\\
								\clm_\alpha (c)
								& \textrm{ if } \alpha \textrm{ is a root of unity.}
							\end{cases}
						\end{equation*}
						It remains to notice that $v_\alpha\left( h \right) \geq 0$ (since $h$
						is a polynomial) to get $\clm_\alpha (c) \geq 0$ for every anxious $\alpha$.
					\item[(i) $\Rightarrow$ (ii)]
						First note that if $\frac{h_0\left( z^k \right)}{h_0\left( z \right)}c\left(
						z \right)$ is precalm for some polynomial $h_0\left( z \right) \in \ck\left[ z \right]$, then
						so is $c\left( z \right)$ -- we can construct the required polynomial for
						$c$ by multiplying the polynomial $h$ we get for $\frac{h_0\left( z^k \right)}{h_0\left( z
						\right)}c\left( z \right)$ by $h_0$. This simple fact allows
						us to construct the required polynomial step by step, in each step
						eliminating one anxious pole of $c$.

						First pick an anxious $\alpha$ at which $c$ has a pole such that either $\alpha$ is
						a root of unity or $c$ has no poles at $\alpha^{k^n}$ for $n > 1$. If no such
						$\alpha$ exists, then $c$ definitely
						cannot have any poles at anxious roots of unity and if it had a pole
						at $\beta$ which is not a root of unity, then it would have a pole at arbitrarily large
						powers of $\beta$, which is impossible for a rational function. Therefore $c$
						is calm, so also precalm, and the proof is finished.

						Since $\alpha$ is anxious, we know from \pointref{lem:it:calmness} that
						$\clm_\alpha c \geq 0$, but $c$ has
						a pole at $\alpha$, so $v_\alpha\left( c \right) < 0$. This means there
						exists an $n \geq 1$ such that $v_{\alpha^{k^n}}\left( c \right) > 0$. We put
						\begin{equation*}
							h\left( z \right) = \prod_{i = 1}^n \left( \alpha^{k^i} - z \right).
						\end{equation*}
						Now $\frac{h\left( z^k \right)}{h\left( z \right)}$ has only a single pole at
						$\alpha^{k^n}$, a root at $\alpha$ and other, unimportant roots. In fact,
						the root
						of the $i^{\textrm{th}}$ factor of $h\left( z \right)$ cancels out with one
						of the roots of the $\left( i+1 \right)^{\textrm{th}}$ factor of $h\left( z^k
						\right)$, while all the other roots are different from the roots of $h\left(
						z \right)$. The last fact is obvious if $\alpha$ is not a root of unity.
						If $\alpha$ is a root of unity, then it has order coprime to $k$
						(since it is anxious) and therefore $\zeta_k\alpha^{k^i} \neq
						\alpha^{k^j}$ (where $\zeta_k$ is a primitive $k^\textrm{th}$ root of
						unity), because the left hand side is a root of unity of order not
						coprime to $k$. This leaves only the root of the last factor of $h\left( z
						\right)$ and first factor of $h\left( z^k \right)$ corresponding to the
						pole and the root mentioned earlier.
						Therefore $\frac{h\left( z^k \right)}{h\left( z \right)} c\left( z
						\right)$ either has no
						pole at $\alpha$ or has a pole at $\alpha$ of order one less that $c$.
						Furthermore, the total number of its poles, counted with multiplicities, is one
						lower than that of $c$, since the pole at $\alpha^{k^n}$ is reduced by the
						root at this very point.
						It remains to prove that the function $\frac{h\left( z^k \right)}{h\left( z \right)}
						c\left( z \right)$ satisfies \pointref{lem:it:calmness}, which allows us
						to conclude by induction over the
						total number of anxious poles. This is a
						simple application of Lemma \ref{lem:actionEffect} -- $h\left( z
						\right)$ has positive valuations only at numbers of the form
						$\alpha^{k^i}$ for $1 \leq i \leq n$, but the calmnesses
						$\clm_{\alpha^{k^i}}c\left( z \right)$ at these points have to be
						positive, since $v_{\alpha}\left( c\left( z \right) \right)$ was
						negative. \qedhere
				\end{description}
			\end{proof}
		\subsection{Regularity criterion for order one Mahler series}
		 In this subsection we prove some equivalent conditions for the regularity of a power
			series satisfying an equation of the form \eqref{eq:orderOne}.

			\begin{theorem}\label{thm:orderOneCriteria}
				If a formal series $f \in \ck\left[ \left[ z \right] \right]$ satisfies a
				Mahler equation of the form
				\begin{equation*}
					f\left( z \right) = c_1\left( z \right)f\left( z^k \right),
				\end{equation*}
				then the following conditions are equivalent
				\begin{enumerate}[(i)]
					\item $f$ is regular,
					\item\label{thm:orderOneCriteria:it:isPrecalm} $c_1$ is precalm,
					\item \label{thm:orderOneCriteria:it:nonnegativeCalmness} $\clm_\alpha (c_1) \geq 0$ for every anxious $\alpha$.
				\end{enumerate}
			\end{theorem}

			Before we prove this theorem it is worth noting that it provides an
			alternative proof of the characterisation of regular rational functions.
			Allouche and Shallit \cite[Theorem 3.3]{AlloucheShallit92} have shown the
			following.
			\begin{theorem}
				A rational function $f\left( z \right) = \frac{p\left( z \right)}{q\left( z
				\right)}$ is regular if and only if $q\left( z \right)$ has roots only at
				roots of unity.
			\end{theorem}
			We prove this result using Theorem \ref{thm:orderOneCriteria}, while the
			original proof used explicit computation and basic properties of regular
			sequences.
			\begin{proof}
				Note that $f\left( z \right)$ satisfies the Mahler equation
				\begin{equation*}
					f\left( z \right) = c\left( z \right) f\left( z^k
					\right)
				\end{equation*}
				with $c\left( z \right) = \frac{p\left( z \right)q\left( z^k \right)}{q\left(
				z \right)p\left( z^k \right)}$. The calmness
				\begin{equation*}
					\clm_{\alpha} \left( c\left( z \right) \right) = \clm_{\alpha}
					\frac{p\left( z \right)}{p\left( z^k \right)} + \clm_{\alpha}
					\frac{q\left( z^k \right)}{q\left( z \right)}
				\end{equation*}
				for any anxious $\alpha$. We can
				show that this expression is nonegative for all anxious $\alpha$ if and only
				if $q\left( z \right)$ has roots only at roots of unity. Indeed, the first
				term is always nonegative and so is the second if $q\left( z \right)$ has no
				roots at numbers different from roots of unity. In fact, this is a
				straightforward application of Lemma \ref{lem:valuationUnderMahler}. Assume that $q\left( z \right)$ has a
				root at a number $\alpha$ that is not a root of unity. Then $q\left( z
				\right)$ has a root at $\alpha_0 = \alpha^{k^i}$, $i \geq 0$, such that $q\left( z \right)$ has no root
				at $\alpha_0^{k^i}$ for any $i \geq 1$. The calmness $\clm_{\alpha_0}c\left(
				z \right) < 0$, because the second term is negative while the first cannot be
				positive since $p\left( z \right)$ is coprime to $q\left( z \right)$.
				An application of Theorem \ref{thm:orderOneCriteria} finishes the proof.
			\end{proof}

			\begin{proof}[Proof of Theorem \ref{thm:orderOneCriteria}]
				The equivalence of conditions
				\pointref{thm:orderOneCriteria:it:isPrecalm} and
				\pointref{thm:orderOneCriteria:it:nonnegativeCalmness} follows immediately from Lemma \ref{lem:precalm}.
				We will again start with the simpler of the remaining implications. The
				opposite implication generalises the work of Dumas \cite[Proposition 54]{Dumas93}.
				\begin{description}
					\item[(ii) $\Rightarrow$ (i)] Since $c_1$ is precalm, we have a $h \in \ck\left[ z
						\right]$ such, that $\frac{h\left( z^k \right)}{h\left( z \right)}c_1\left(
						z \right)$ is calm. This, together with the definition of $f$, means that the series $\frac{f\left( z
						\right)}{h\left( z \right)}$ satisfies the Mahler equation
						\begin{equation*}
							\frac{f\left( z \right)}{h\left( z \right)} = c_1'\left( z \right)
							\frac{f\left( z^k \right)}{h\left( z^k \right)}
						\end{equation*}
						with $c_1'\left( z \right) = \frac{h\left( z^k \right)}{h\left( z
						\right)}c_1\left( z \right)$ calm. From Theorem
						\ref{thm:DumasReguralityCriterion} we know that series satisfying such a
						Mahler equation are regular. But polynomials are also regular and regular
						power series form a ring \cite[Corollary 3.2]{AlloucheShallit92}, so
						$f\left( z \right) = h\left( z \right)\frac{f\left( z \right)}{h\left( z
						\right)}$ is also regular.
					\item[(i) $\Rightarrow$ (iii)] We will prove that if 
					 $\clm_\alpha (c_1) < 0$ for some anxious $\alpha$, then $f$ cannot be
						regular. Using the Mahler equation \eqref{eq:orderOne} and Lemma
						\ref{lem:lambdaProp}.\pointref{lambdaProd} we can for any
						number $L+1$ of indices $r_i \in \left\{ 0, \dots, k - 1 \right\}$ write
						\begin{equation}\label{eq:lotsaLambdas}
							\Lambda_{r_{L+1}} \dots \Lambda_{r_1} f(z) = \left( \Lambda_{r_L} \dots
							\Lambda_{r_1} \prod_{l = 0}^{L} c_1\left( z^{k^l} \right)
							\right) f(z).
						\end{equation}
						We will now show that the $\ck$-vector space $V$ generated by these
						expressions is not finitely dimensional -- by Lemma
						\ref{lem:lambdaRegularity} this implies that $f$ is not regular.

					 We know that $\clm_\alpha c_1 < 0$ for some anxious $\alpha$. Let us
						define a set of indices $R_\alpha$. If $\alpha$ is a root of
						unity, then $R_\alpha = \left\{ im - 1 \colon i \geq 0 \right\}$ where $m$ is the
						smallest number such that $\alpha^{k^m} \neq \alpha$.
						If $\alpha$ is not a root of unity, then $R_\alpha$ consist of all natural numbers
						greater than $m$, where $m$ is the largest number for which
						$v_{\alpha^{k^m}}\left( f \right) \neq 0$. In both cases the set
						$R_\alpha$ has two important properties -- it contains arbitrarily large
						numbers and the valuation
						$v_\alpha\left( \prod_{l = 0}^L c_1\left( z^{k^l} \right) \right) < 0$ for
						any $L \in R_\alpha$. The
						latter property holds for $\alpha$ not a root of unity, because the
						valuation is equal to $\clm_\alpha c_1$ and for $\alpha$ a root of unity
						because the valuation is
						equal to a natural multiple of $\clm_\alpha c_1$.

						In \eqref{eq:lotsaLambdas} we put no restrictions on the choice of $r_i$.
						Therefore, by repeated applications of Lemma
						\ref{lem:polePerserving}, we can choose $r_i$ so that
						\begin{equation*}
							v_\alpha\left( \Lambda_{r_{L+1}}\dots\Lambda_{r_1} \prod_{l = 0}^L
							c_1\left( z^{k^{l + L + 1}} \right) \right) \leq v_\alpha\left( \prod_{l =
							0}^L c_1\left( z^{k^l} \right) \right)
						\end{equation*}
						for $L \in R_\alpha$. By Lemma \ref{lem:valuationUnderMahler}, this can be
						also written as
						\begin{equation}\label{eq:nongrowingValuation}
							v_{\alpha^{k^{L+1}}} \left( \Lambda_{r_{L+1}} \dots  \Lambda_{r_1} \prod_{l = 0}^L
							c_1\left( z^{k^l} \right) \right) \leq v_\alpha\left(\prod_{l =
							0}^L c_1\left( z^{k^l}\right) \right).
						\end{equation}
						Note that $v_\alpha\left(\prod_{l = 0}^L c_1\left( z^{k^l}\right) \right)
						< 0$ for $L \in R_\alpha$.

						By the form of the Mahler equation \eqref{eq:orderOne} and Lemma
						\ref{lem:lambdaProp}.\pointref{lambdaProd}, we can see that the vector
						space $V$ contains only rational functions multiplied by
						the power series $f\left( z \right)$. It is therefore a subset of a one
						dimensional $\ck\left( z \right)$-vector space, and we identify it with a
						subset of $\ck\left( z \right)$ via the map $V \ni g\left( z
						\right)f\left( z \right) \mapsto g\left( z \right) \in \ck\left( z
						\right)$. This allows us to talk about
						valuations of elements of $V$ via the formula $v_\alpha\left( g\left( z
						\right)f\left( z \right)	\right) = v_\alpha\left( g\left( z \right)
						\right)$, $g\left( z \right) \in \ck\left( z \right)$.

						Using equations \eqref{eq:lotsaLambdas} and \eqref{eq:nongrowingValuation}
						and varying $L \in R_\alpha$ allows us to look
						at valuations of elements in the $\ck$-vector space $V$.
						There are two cases to consider:
						\begin{description}
							\item[$\alpha$ is not a root of unity] In this case there are infinitely many
								valuations $v_{\alpha^{k^{L + 1}}}$ for which there exist elements of
								$V$ with negative valuation. However, if the vector space $V$ was finitely
								dimensional, then each of its
								generators would have negative valuation only at finitely many points,
								which gives a contradiction.
							\item[$\alpha$ is a root of unity] In this case, by the form of
								$R_\alpha$, $\alpha^{k^{L+1}} = \alpha$ for $L \in R_\alpha$ and the
								product on the right hand side of \eqref{eq:nongrowingValuation} can
								have arbitrarily low valuation. Thus there are elements in the vector
								space $V$ with arbitrarily low valuation with respect to $v_\alpha$, so
								as before $V$ cannot be finitely dimensional.
						\end{description}
				\end{description}
			\end{proof}
	\section{Calmnesses and the precalmness property}\label{sect:calmnesses}
	 In this section we explore the relation between calmnesses of a sequence of
		functions and its precalmness. We prove a general version of Lemma
		\ref{lem:precalm}.
		\begin{theorem}\label{thm:precalmness}
			A sequence of rational functions $\left( c_i \right)_{i = 1}^n$ is precalm if
			and only if for every anxious $\alpha$ and for every $(a_i) \in A_\alpha$ we
			have $\clm_{\alpha, (a_i)} (c_i) \geq 0$.
		\end{theorem}
		We will prove this result in two steps, first considering only the case when
		$\alpha$ is a root of unity and later finishing the proof for other anxious
		numbers.
		\subsection{Eliminating poles at roots of unity}
			We will first focus on roots of unity and show that under the assumption that
			all the calmnesses are nonnegative, we can eliminate poles at
			such numbers by acting on the sequence of coefficients with a certain
			polynomial. We introduce a	few more definitions and lemmas that are required to treat
			this case and will not be used elsewhere.
			\begin{definition}
				Let $\alpha$ be an anxious root of unity. Set $h_\alpha\left( z
				\right) = z - \alpha$.
			\end{definition}
			From now on, fix an anxious root of unity $\alpha$ and let $0 < m \in \N$ be the
			smallest positive integer such that $\alpha^{k^m} = \alpha$. Such an integer
			exists because anxious roots of unity are exactly roots of unity of order
			coprime to $k$ and furthermore $\left( \alpha^{k^i} \right)^{k^m} =
			\alpha^{k^i}$ for any $i \in \N$. Let also $\mathcal{A} = \bigsqcup_{i =
			0}^{m - 1} A_{\alpha^{k^i}}$ be the disjoint union of all the sequences corresponding to
			calmnesses for all $\alpha^{k^i}$, $0 \leq i \leq m - 1$.
			\begin{lemma}\label{lem:simpleRootsOperations}
				Let $\left( c_i \right)$ be a sequence of rational functions. Then for $\left(
				\widetilde{c}_i \right) = \left( h_{\alpha^{k^{j_0}}} \right)^*\left( c_i	\right)$ we have
				\begin{equation*}
					v_{\alpha^{k^j}}\left( \widetilde{c}_i \right) = v_{\alpha^{k^j}}\left(
					c_i \right) +
					\begin{cases}
						-1 & \textrm{ if } j \equiv j_0 \pmod{m} \textrm{ and } i \not\equiv \left( j_0 - j
						\right) \pmod{m}\\
						1 & \textrm{ if } i \equiv \left( j_0 - j \right) \pmod{m} \textrm{ and }
						j \not\equiv j_0 \pmod{m}\\
						0 & \textrm{ otherwise.}
					\end{cases}
				\end{equation*}
			\end{lemma}
			\begin{proof}
				By Lemma \ref{lem:valuationUnderMahler}, it suffices to check that
				\begin{equation*}
					\begin{split}
					v_{\alpha^{k^j}}\left( \frac{h_{\alpha^{k^{j_0}}}\left( z^{k^i}
					\right)}{h_{\alpha^{k^{j_0}}}\left( z \right)} \right) &=
					v_{\alpha^{k^{j + i}}}\left( z - \alpha^{k^{j_0}} \right) -
					v_{\alpha^{k^j}}\left( z - \alpha^{k^{j_0}} \right) = \\ &=
					\begin{cases}
						-1 & \textrm{ if } j \equiv j_0 \pmod{m} \textrm{ and } i \not\equiv \left( j_0 - j
						\right) \pmod{m}\\
						1 & \textrm{ if } i \equiv \left( j_0 - j \right) \pmod{m} \textrm{ but }
						j \not\equiv j_0 \pmod{m}\\
						0 & \textrm{ otherwise.}
					\end{cases}
					\end{split}
				\end{equation*}
				The latter equality stems from the fact that the first term is nonzero (and
				so equal to one) when $(j + i) \equiv j_0 \pmod{m}$, the second term is
				nonzero when $j \equiv j_0 \pmod{m}$ and if both these conditions are true,
				then the terms cancel out.
			\end{proof}
			The following definition will be needed just for the next few lemmas.
			\begin{definition}
				For a sequence of rational functions $\left( c_i \right)$, consider some
				term $v_{\alpha^{k^a}}\left( c_{b} \right)$. Among all the calmnesses
				$\clm_{\alpha, \left( a_i \right)}\left( c_i \right)$
				associated with sequences $\left( a_i \right) \in \mathcal{A}$ that contain
				the term $v_{\alpha^{k^a}}\left( c_{b} \right)$ as one
				of their summands (\ie  sequences $\left( a_i \right) \in
				A_{\alpha^{k^l}}$ ($0 \leq l \leq m - 1$) such that for some $y$,
				$\alpha^{k^{l \overline{a}_y}} = \alpha^{k^a}$ and $c_b = c_{a_y}$), there
				are some having minimal value. We will call them the \emph{minimal calmnesses containing
				$v_{\alpha^{k^a}}\left( c_{b} \right)$} and we will denote the
				set of sequences $\left( a_i \right)$ associated with them by $\minclm_{a,
				b}\left( c_i
				\right) \subset \mathcal{A}$.
			\end{definition}
			\begin{proposition}\label{prop:rootsElimination}
				Let the sequence $\left( c_i(z) \right)_{i =1}^n$, $c_i\left( z \right)
				\in \ck\left(	z \right)$ be of length $n \geq 1$. If for every $\left( a_i
				\right) \in \mathcal{A}$,	$\clm_{\left( a_i \right)}\left( c_i \right) \geq
				0$, then there exists a polynomial $0 \neq h \in \ck\left[ z \right]$ such
				that no term of the sequence $h^* \left( c_i\left( z \right) \right)$ has a
				pole at $\alpha^{k^i}$ for $i \geq 0$.
			\end{proposition}
			We will prove this proposition in a series of lemmas.
			\begin{lemma}\label{lem:simpleRootsElimination}
				Let the sequence $\left( c_i(z) \right)_{i =1}^{m}$, $0 \neq c_i\left( z \right)
				\in \ck\left(	z \right)$ be of length $m$ and assume that it contains no
				zero terms. In such a case, if for every $a \geq 0$, $b \geq 1$ and $\left(
				a_i \right) \in \mathcal{A}$,	$\clm_{\left( a_i \right)}\left( c_i \right)
				\geq 0$ and for every $\left( a_i \right) \in
				\minclm_{a,b}\left( c_i \right)$ the calmness $\clm_{\left( a_i \right)}\left(
				c_i \right) = 0$, then there exists a polynomial
				$0 \neq h \in \ck\left[ z \right]$ such that the sequence $h^*
				\left( c_i\left( z \right) \right)$ is $\alpha^{k^i}$-calm for any $i
				\geq 0$.
				\begin{proof}
					We will construct the required polynomial in steps. In every step we will
					choose a polynomial to act on the sequence and substitute the result for the
					original sequence. We will take $h$ to be the product of these polynomials.
					Every step will reduce the sum $\sigma = -\sum_{i,j \colon v_{\alpha^{k^i}}\left( c_j
					\right) < 0} v_{\alpha^{k^i}}\left( c_j \right)$. This sum is always
					nonnegative and if it is $0$, then no term of the sequence has a pole at
					$\alpha^{k^i}$ for $i \geq 0$. By Lemma \ref{lem:actionEffect}, calmness
					$\clm_{\left( a_i \right)}\left( c_i \right)$ remains unchanged upon the
					action of $\ck\left( z \right)$, and so it is enough to	prove that if
					$\sigma > 0$, we can find a polynomial $h$ such that acting by $h$ produces
					a sequence with a smaller value of $\sigma$.

					Assume $\sigma > 0$ and pick $0 \leq j \leq m$ such that $v_{\min}\left( j
					\right) := \min_{i}\left\{ v_{\alpha^{k^j}}\left( c_i \right) \right\}$ is
					minimal among all $j$. Define $S_j := \left\{ 1
					\leq i \leq m	\colon v_{\alpha^{k^j}}\left( c_{i} \right) =
					v_{\min}\left( j \right) \right\}$. The choice of $j$ satisfies
					$v_{\min}\left( j \right) < 0$ since $\sigma > 0$. The set $S_j$ is neither
					empty (by the definition of $v_{\min}\left( j \right)$) nor does it contain all possible indices (for
					example $m \not\in S_j$, because the valuation $v_{\alpha^{k^j}}\left( c_m
					\right)$ is equal to the calmness $\clm_{\alpha^{k^j}, (m)}\left( c_i
					\right)$ associated with the one element sequence $(m)$ and therefore
					$v_{\alpha^{k^j}}\left( c_m \right) \geq 0$).
					Pick an $i_0 \in S_{j}$. For any $i_1 \in \left\{ 1, \dots, m
					\right\}\setminus S_{j}$, we can pick a sequence $\left( a_i \right) \in
					\minclm_{j, i_1}\left( c_i \right)$. By assumption $\clm_{\left( a_i
					\right)}\left( c_i \right) = 0$. Consider the sequence $\left( b_i
					\right)$ of length equal to the length of $\left( a_i \right)$ plus one and
					such that
					\begin{equation*}
						b_i = \begin{cases}
							a_i & \textrm{ for } i < i'\\
							i_0 & \textrm{ for } i = i'\\
							i_2 \equiv i_1 - i_0 \pmod{m} & \textrm{ for } i = i' + 1\\
							a_{i-1} & \textrm{ for } i > i' + 1
						\end{cases}
					\end{equation*}
					with $i'$ such that $a_{i'} = i_1$. If we treat $\left( b_i \right)$ as a
					sequence associated with the same anxious number as $\left( a_i \right)$, we
					get
					\begin{equation*}
						\begin{split}
						\clm_{\left( b_i \right)}\left( c_i \right) &= \clm_{\left( a_i
						\right)}\left( c_i \right) - v_{\alpha^{k^{j}}}\left( c_{i_1} \right) +
						v_{\alpha^{k^{j}}}\left( c_{i_0} \right) + v_{\alpha^{k^{j +
						i_0}}}\left( c_{i_2} \right) = \\ &= - v_{\alpha^{k^{j}}}\left( c_{i_1} \right) +
						v_{\alpha^{k^{j}}}\left( c_{i_0} \right) + v_{\alpha^{k^{j +i_0}}}\left(
						c_{i_2} \right).
						\end{split}
					\end{equation*}
					Again by assumption $\clm_{\left( b_i \right)}\left( c_i \right) \geq 0$,
					which gives us
					\begin{equation*}
						v_{\alpha^{k^{j + i_0}}}\left( c_{i_2} \right) \geq
						v_{\alpha^{k^{j}}}\left( c_{i_1} \right) - v_{\alpha^{k^{j}}}\left(
						c_{i_0} \right) > 0,
					\end{equation*}
					where the last inequality holds because $i_1 \not\in S_{j}$.
					Therefore, for any $i_1 \not\in S_{j}$ we have $v_{\alpha^{k^{j +
					i_0}}}\left( c_{i_2} \right) > 0$. On the other hand, for $i_0 \neq m$,
					by applying the property $\clm_{\left( a_i \right)}\left( c_i
					\right) \geq 0$ to the length two sequence $\left( i_0, m - i_0 \right) \in
					A_{\alpha^{k^{j}}}$, we obtain
					$v_{\alpha^{k^{j + i_0}}}\left( c_{m - i_0} \right) > 0$.
					Applying the above reasoning for all $i_0 \in S_{j}$
					provides us with inequalities
					\begin{equation}
					\label{ieq:positiveColumns}
						v_{\alpha^{k^{j + i_0}}}\left( c_{i_2} \right) > 0,
					\end{equation}
					where $i_2 = i_1 - i_0 \pmod{m}$ for some $i_1 \not\in S_{j}$ or $i_2 = m -
					i_0$. Recall further that $m \not\in S_{j}$.

					Take
					$h = \prod_{i_0 \in S_{j}} h_{\alpha^{k^{j + i_0}}}$. We claim that the
					action of $h$ on $\left( c_i \right)$ lowers $\sigma$. This is a
					straightforward application of Lemma \ref{lem:simpleRootsOperations}. In
					fact, the
					valuations $v_{\alpha^{k^{j}}}\left( c_{i_0} \right)$ for $i_0 \in
					S_{j}$ are all negative and acting on them by $h$ adds one to each of them, so it
					is enough to note that this action does not cause any of the remaining valuations to
					drop below
					zero. We pick $i_0 \in S_{j}$ and check that this is the case for
					valuations of the form $v_{\alpha^{k^{j + i_0}}}\left( c_{i_2} \right)$. The
					inequalities of the form \eqref{ieq:positiveColumns} and the fact that
					acting by the
					polynomial $h$ lowers valuations by at most $1$ (Lemma
					\ref{lem:simpleRootsOperations}) allows us to check this
					only for $i_2 \equiv i_1 - i_0 \pmod{m}$, where $i_1 \in S_{j} \setminus \left\{ i_0
					\right\}$. However, for any such $i_2$, the action of $h$ does not actually change
					this valuation since by Lemma \ref{lem:simpleRootsOperations} it is lowered by
					$1$ by the action of $h_{\alpha^{k^{j + i_0}}}$ but is raised by $1$ by the
					action of $h_{\alpha^{k^{j + i_1}}}$, because $j + i_1 - \left( j + i_0
					\right) \equiv i_1 - i_0 \equiv i_2 \pmod{m}$. This ends the step and the
					proof.
				\end{proof}
			\end{lemma}
		 In the next two lemmas we will weaken the assumptions of Lemma
			\ref{lem:simpleRootsElimination}.

			\begin{lemma}\label{lem:nonzeroMinclmRootsElimination}
				Let the sequence $\left( c_i(z) \right)_{i =1}^{m}$, $0 \neq c_i\left( z \right)
				\in \ck\left(	z \right)$ be of length $m$ and contain no zero terms. In such
				a case, if for every $\left( a_i \right) \in \mathcal{A}$,	$\clm_{\left( a_i
				\right)}\left( c_i \right) \geq 0$, then there exists a polynomial
				$0 \neq h \in \ck\left[ z \right]$ such that the sequence $h^*
				\left( c_i\left( z \right) \right)$ is $\alpha^{k^i}$-calm for any $i
				\geq 0$.
				\begin{proof}
					Consider $\left( a_i \right) \in \minclm_{a,b}\left( c_i \right)$. If
					$\clm_{\left( a_i \right)}\left( c_i \right) \neq 0$, substitute the sequence
					$\left( c'_i \right)$ for $\left( c_i \right)$, where
					\begin{equation*}
						c'_i\left( z \right) = \begin{cases}
							\frac{c_i\left( z \right)}{\left( z - \alpha^{k^a} \right)^{
							\clm_{\left( a_i \right)}\left( c_i \right)}} & \textrm{ for } i=b\\
							c_i\left( z \right) & \textrm{ otherwise.}
						\end{cases}
					\end{equation*}
					The new sequence now satisfies $\clm_{\left( a_i \right)}\left(
					c'_i \right) = 0$ for all $a_i \in \minclm_{a,b}\left( c'_i \right)$ while
					still satisfying $\clm_{\left( a_i \right)}\left( c'_i \right) \geq 0$ for
					all $a_i \in \mathcal{A}$. We can repeat this procedure as long as for any
					$a \geq 0, b \geq 1$ we have $\clm_{\left( a_i \right)} \left( c_i
					\right) \neq 0$ for $\left( a_i \right) \in \minclm_{a,b}\left( c_i
					\right)$. Since all $c_i \neq 0$, we only need to repeat this procedure
					finitely many times. The sequence obtained after applying
					these procedures satisfies the assumptions of Lemma
					\ref{lem:simpleRootsElimination}. Therefore, there exists a polynomial $h$ such
					that $h^*\left( c'_i \right)$ is $\alpha^{k^i}$-calm for any
					$i \geq 0$. But $v_{\alpha^{k^i}}\left( c_i \right) \geq
					v_{\alpha^{k^i}}\left( c'_i \right)$, so $h^*\left( c_i \right)$ is also
					$\alpha^{k^i}$-calm.
				\end{proof}
			\end{lemma}
			\begin{lemma}\label{lem:zeroTermsRootsElimination}
				Let the sequence $\left( c_i(z) \right)_{i =1}^{m}$, $c_i\left( z \right)
				\in \ck\left(	z \right)$ be of length $m$. In such a
				case, if for every $\left( a_i \right) \in \mathcal{A}$,	$\clm_{\left( a_i
				\right)}\left( c_i \right) \geq 0$, then there exists a polynomial
				$0 \neq h \in \ck\left[ z \right]$ such that the sequence $h^*
				\left( c_i\left( z \right) \right)$ is $\alpha^{k^i}$-calm for any $i
				\geq 0$.
				\begin{proof}
					Let $Z = \left\{ i \colon c_i\left( z \right) \neq 0 \right\}$ and $p =
					\sum_{i \in Z, 0 \leq j \leq m-1} \left| v_{\alpha^{k^j}}\left( c_i \right)
					\right|$. Consider	the sequence of rational functions $\left( c'_i \right)$ where
					\begin{equation*}
						c'_i\left( z \right) = \begin{cases}
							c_i\left( z \right) & \textrm{ if } i \in Z\\
							\prod_{j} \left( z - \alpha^{k^j} \right)^p & \textrm{ otherwise.}
						\end{cases}
					\end{equation*}

					We claim that $\left( c'_i \right)$ satisfies the assumptions of Lemma
					\ref{lem:nonzeroMinclmRootsElimination}. Since there are no zero terms in
					the sequence $\left( c'_i \right)$, it is enough to show that
					$\clm_{\left( a_i \right)}\left( c'_i \right) \geq 0$ for all $\left( a_i
					\right) \in \mathcal{A}$. Assume there exists $\left( a_i \right) \in
					\mathcal{A}$ such that $\clm_{\left( a_i \right)}\left( c'_i \right) < 0$.
					We can assume there are no sequences of shorter length than $\left( a_i \right)$
					satisfying this inequality. In this case, no two terms of $\left(
					\overline{a}_i \right)$ are equal modulo $m$. Indeed, if there existed $i_0
					\neq i_1$ such that $\overline{a}_{i_0} \equiv \overline{a}_{i_1} \pmod{m}$
					we could consider the sequence $\left( a_{i - i_0} \right)_{i = i_0}^{i_1 -
					1} \in A_{\alpha^{k^{\overline{a}_{i_0}}}}$ and either its
					calmness would be negative (contrary to the assumption that $\left( a_i
					\right)$ was the shortest sequence with this property) or nonnegative, but then the
					sequence $\left( a'_i \right)$ with
					\begin{equation*}
						a'_i = \begin{cases}
							a_i & \textrm{ for } i < i_0\\
							a_{i - i_0 + i_1} & \textrm{ for } i_0 \leq i
						\end{cases}
					\end{equation*}
					would be shorter than $\left( a_i \right)$ and also have a negative
					calmness. Now
					\begin{equation*}
						\clm_{\left( a_i \right)}\left( c'_i \right) = \sum_{a_i \in Z}
						v_{\alpha^{k^{\overline{a}_i}}}\left( c'_{a_i}\left( z \right) \right) +
						\sum_{a_i \not\in Z} v_{\alpha^{k^{\overline{a}_i}}}\left( c'_{a_i}\left( z
						\right) \right).
					\end{equation*}
					If the latter sum contains no terms, then $\clm_{\left( a_i \right)}\left(
					c'_i \right)$ is nonnegative since
					$\clm_{\left( a_i \right)}\left( c'_i \right) = \clm_{\left( a_i
					\right)}\left( c_i \right) \geq 0$. If the latter sum contains at least one term, then
					the former sum is necessarily at least $-p$ by the definition of $p$, so adding $p$ to it
					makes it nonnegative.

					Therefore $\left( c'_i \right)$ satisfies the assumptions of Lemma
					\ref{lem:nonzeroMinclmRootsElimination} and there exists a polynomial $h$ such
					that $h^*\left( c'_i \right)$ is $\alpha^{k^i}$-calm for any
					$i \geq 0$, but as in the previous proof $v_{\alpha^{k^i}}\left( c_i \right) \geq
					v_{\alpha^{k^i}}\left( c'_i \right)$, so $h^*\left( c_i \right)$ is also
					$\alpha^{k^i}$-calm.
				\end{proof}
			\end{lemma}
			\begin{proof}[Proof of Proposition \ref{prop:rootsElimination}]
				Let $s_{i, j} = \min_{r\geq 0} v_{\alpha^{k^j}}\left( c_{r m + i} \right)$. Consider the sequence of rational functions $\left( c'_i
				\right) = \left( \prod_j \left( z - \alpha^{k^j} \right)^{s_{i, j}} \right)_{i
				= 1}^{m}$. This sequence satisfies the assumptions of Lemma
				\ref{lem:zeroTermsRootsElimination} since $\clm_{\left( a_i \right)}\left(
				c'_i \right) = \clm_{\left( a'_i \right)}\left( c_i \right) \geq 0$ for
				$\left( a'_i \right)$ chosen in such a way that $a'_i = r m + a_i$ for $r$
				chosen so that $v_{\alpha^{k^j}}\left( c_{a_i} \right) = s_{a_i, j}$. We therefore have a polynomial $h$
				such that $h^*\left( c'_i \right)$ is $\alpha^{k^i}$-calm for any
				$i \geq 0$. To see that also $h^*\left( c_i \right)$ is $\alpha^{k^i}$-calm,
				note that by Lemma \ref{lem:simpleRootsOperations} this action
				changes valuations at $\alpha^{k^j}$ of $c_{i + r m}$ by the same amount
				for any $r \in \N$. Since $\left( c'_i \right)$ was constructed in such a way
				that $v_{\alpha^{k^j}}\left( c_{r m + i} \right) \geq v_{\alpha^{k^j}}\left(
				c'_i \right)$, this finishes the proof.
			\end{proof}
		\subsection{The remaining cases}
			\begin{proof}[Proof of Theorem \ref{thm:precalmness}]
				First note that precalm sequences clearly admit only nonnegative calmnesses.
				Indeed, calm sequences trivially have nonnegative calmnesses, and so do
				precalm sequences by Lemma \ref{lem:actionEffect}.
				The converse implication is slightly more challenging.

				Assume $\left( c_i \right)_{i = 1}^n$ is a sequence such that all the
				calmnesses $\clm_{\alpha, \left( a_i \right)}$ are nonnegative. Then, by
				applying Proposition \ref{prop:rootsElimination} to all anxious
				roots of unity at which terms of $\left( c_i \right)$ have poles, we find a
				polynomial $0 \neq h_0 \in \ck\left[ z \right]$ such that $h_0^*\left( c_i
				\right)$ has no poles at anxious roots of unity. Since $h_0^*\left( c_i
				\right)$ has the same calmnesses as $\left( c_i \right)$ (by Lemma
				\ref{lem:actionEffect}) we have reduced the proof to the case when
				$c_i$ have no poles at anxious roots of unity. We assume this henceforth.

				We will construct a polynomial $0 \neq h \in \ck\left[ z \right]$ such that
				$h^*\left( c_i \right)$ has no poles at anxious numbers $\alpha$ in a number
				of steps. At each step, we will construct a polynomial $\widetilde{h}\in
				\ck[z]$ and substitute $\widetilde{h}^*\left(	c_i \right)$ for the original
				sequence. We will show that this substitution perserves the assumptions and
				that the procedure finishes after a finite number of steps. We will also
				show that the procedure stops only when the sequence is calm. The desired
				polynomial $h$ will be the product of all the polynomials $\widetilde{h}$
				constructed in this way.

				Let $\alpha \in \ck$ be an anxious number such that there exists a
				$c_{i_0}$ such that $v_\alpha \left( c_{i_0}(z) \right) < 0$. In particular
				$\alpha$ is not a root of unity. If no such
				$\alpha$ exists, the sequence is calm and the procedure stops. The
				assumptions of the theorem imply that for the constant sequence $(i_0) \in
				A_\alpha$ the calmness $\clm_{(i_0)} (c_i) \geq 0$. Therefore, since
				$v_\alpha \left( c_{i_0}(z) \right) < 0$, there	exists $j$ such that
				$v_{\alpha^{k^{j i_0}}} \left( c_{i_0}(z) \right) > 0$. Let $j_0$ be the
				smallest such $j$ and let $\widetilde{h}(z) = \prod_{i = 1}^{j_0} \left( z
				- \alpha^{k^{i_0 i}} \right)$.

				By Lemma \ref{lem:actionEffect}, for any anxious $\beta$ and $(b_i) \in
				A_\beta$ we have $\clm_{(b_i)} \widetilde{h}^* (c_i) = \clm_{(b_i)} (c_i) -
				v_\beta\left( \widetilde{h}(z) \right)$. By	the form of $\widetilde{h}$, the
				term $v_\beta\left( \widetilde{h}\left( z \right) \right)$ is nonzero only
				when $\beta = \alpha^{k^{i_0 l}}$ for some integer $1
				\leq l \leq j_0$. Let $(b'_i) \in A_\alpha$ be defined as
				\begin{equation*}
					b'_i =
					\begin{cases}
						i_0 & \textrm{ if } i \leq l. \\
						b_{i - l} & \textrm{ otherwise.}
					\end{cases}
				\end{equation*}
				By the choice of $j_0$, we have $\sum_{i = 0}^{l - 1} v_{\alpha^{k^{i_0 i}}}\left(
				c_{i_0} (z) \right) \leq -1 = -v_{\alpha^{k^{i_0 l}}}\left( \widetilde{h}(z)
				\right)$. Using this, it is immiediate to see that
				\begin{equation*}
					\clm_{(b_i)} (c_i) - v_{\alpha^{k^{i_0 l}}}\left( \widetilde{h}(z) \right)
					\geq \clm_{(b_i)} (c_i) + \sum_{i = 0}^{l - 1} v_{\alpha^{k^{i_0 i}}}\left(
					c_{i_0}(z) \right) = \clm_{(b'_i)} (c_i) \geq 0.
				\end{equation*}
				Therefore, the sequence $\widetilde{h}^*\left( c_i \right)$ still satisfies
				the assumptions of the theorem and we can continue the procedure.

				At the beginning of the procedure fix the set of anxious numbers
				\begin{equation*}
					\mathcal{B} = \left\{ \alpha_r \colon \exists_i
					v_{\alpha_r}\left( c_i\left( z \right) \right) < 0 \textrm{ and } \forall_{j > 0}
					\left( \alpha^{k^j} = \alpha_r \Rightarrow \forall_i v_{\alpha^{k^j}}\left(
					c_i(z) \right) \geq 0 \right) \right\}.
				\end{equation*}
				These are anxious numbers $\alpha_r$ at which at least one $c_i$ has a pole
				and $\alpha_r$ is not a $k^j$$^\mathrm{th}$ power of another pole of some
				$c_i$. Since all $c_i$ are rational functions and there are only finitely
				many of them, this set is finite. Furthermore among the calmnesses
				$\clm_{(a_i)} (c_i)$, $(a_i) \in A_{\alpha_r^{k^l}}$, $\alpha_r \in
				\mathcal{B}$, $l \geq 0$ only finitely many are nonzero. It remains to note
				that each step of the procedure lowers at least one of those calmnesses by
				one and does not increase the remaining ones. This can only happen finitely many times, since they are all
				nonnegative. Specifically, the $\alpha$ chosen in each step of the procedure is
				equal to $\alpha_r^{k^l}$ for some $\alpha_r \in \mathcal{B}$ and $l \geq
				0$, so by Lemma \ref{lem:actionEffect} all the corresponding calmnesses are lowered.

				The two above paragraphs prove that the procedure must end in finitely many
				steps, which ends the proof.
			\end{proof}
	\section{Becker's conjecture}\label{sect:conjectures}
		This whole work was inspired by a conjecture of Becker \cite[Corollary 1,
		remark 2]{Becker92}. In this section we investigate some alternative
		formulations of that conjecture.
		\begin{conjecture}[Becker]\label{conj:Becker}
			A power series $f \in \ck\left[ \left[ z \right] \right]$ is $k$-regular
			if and only if there exists a $k$-regular rational function $0 \neq h\left( z \right)
			\in \ck\left( z \right)$ such that $\frac{f\left(
			z \right)}{h\left( z \right)}$ satisfies a Mahler equation with polynomial
			coefficients.
		\end{conjecture}
		It is useful to note that $k$-regular rational functions are exactly rational
		functions with poles only at $0$ and roots of unity, as shown in \cite[Theorem
		3.3]{AlloucheShallit92}. One implication in Conjecture \ref{conj:Becker}
		(namely, that if such a function exists, then $f$ is regular) has been shown
		by Becker \cite[Corollary 1]{Becker92} -- this is in fact what
		inspired the conjecture. By analogy Theorem \ref{thm:DumasReguralityCriterion}
		might inspire an alternative version of Becker's conjecture.
		\begin{conjecture}[Becker, alternative version]\label{conj:altBecker}
			A power series $f \in \ck\left[ \left[ z \right] \right]$ is $k$-regular
			if and only if there exists a polynomial $0 \neq h\left( z \right)
			\in \ck\left[ z \right]$ such that $\frac{f\left(z \right)}{h\left( z
			\right)}$ satisfies a Mahler equation with a $k$-calm sequence of coefficients.
		\end{conjecture}
		Again, one of the implications in this conjecture follows immiediately from
		Theorem \ref{thm:DumasReguralityCriterion} and the fact that regular power
		series form a ring.

		As it turns out, both these conjectures are equivalent to the following,
		seemingly stronger, conjecture.
		\begin{conjecture}[Becker, final version]\label{conj:finBecker}
			A power series $f \in \ck\left[ \left[ z \right] \right]$ is $k$-regular
			if and only if there exists a polynomial $0 \neq h\left( z \right)
			\in \ck\left[ z \right]$ such that $\frac{f\left(
			z \right)}{h\left( z \right)}$ satisfies a Mahler equation with polynomial
			coefficients.
		\end{conjecture}

		\begin{proof}[Proof of equivalence of Conjectures 1-3]
			Since all the conjectures claim an equivalence of $k$-regularity and a
			certain property of a finite sequence of rational functions $c_i\left( z
			\right)$, it suffices to prove the following lemma.
			\begin{lemma}\label{lem:equivPropForCoeffs}
				Let $\left( c_i\left( z \right) \right)_{i = 1}^n$ be a sequence of
				rational functions. The following conditions are equivalent:
				\begin{enumerate}[(i)]
					\item\label{it:Becker} There exists a rational function $0 \neq h\in
						\ck\left( z \right)$ with poles at most at zero and roots of unity and
						such that $h^*\left( c_i \right)$ is a sequence of polynomials.
					\item\label{it:altBecker} There exists a polynomial $0 \neq h\in \ck\left[ z
						\right]$ such that $h^*\left( c_i \right)$ is $k$-calm.
					\item\label{it:finBecker} There exists a polynomial $0 \neq h\in \ck\left[ z
						\right]$ such that $h^*\left( c_i \right)$ is a sequence of polynomials.
				\end{enumerate}
			\end{lemma}
			\begin{proof}
				The implications \pointref{it:finBecker}$\Rightarrow$\pointref{it:Becker} and
				\pointref{it:finBecker}$\Rightarrow$\pointref{it:altBecker} are obvious, since
				polynomials are $k$-calm and have no poles.
				\begin{description}
					\item[\pointref{it:altBecker}$\Rightarrow$\pointref{it:finBecker}] Choose a
						polynomial $0 \neq h\left( z \right) \in \ck\left[ z \right]$ such that
						the sequence $h^*\left( c_i\left( z \right) \right)$ is $k$-calm. By
						Lemma \ref{lem:prepolynomial}, there exists a polynomial $0 \neq h'\left(
						z \right)\in \ck\left[ z \right]$ such that $\left( h'h \right)^*\left(
						c_i\left( z \right) \right) = h'^*\left( h^*\left( c_i\left(
						z \right) \right) \right)$ is a sequence of polynomials.
					\item[\pointref{it:Becker}$\Rightarrow$\pointref{it:altBecker}] 
						Choose a rational function $h\left( z	\right) = \frac{p\left( z
						\right)}{q\left( z \right)}$ where $p,q \in \ck\left[ z \right]$ are
						polynomials and $q$ has roots only at zero and roots of unity and such that
						$h^*\left( c_i \right)$ is a sequence of polynomials. It is sufficient to
						show that $p^*\left( c_i \right) = q^*\left( h^*\left( c_i \right)  \right)$ is
						$k$-precalm, because then we would have a polynomial $h'\left( z
						\right) \in \ck\left[ z \right]$ such that $h'^*\left( q^*\left( h^*\left(
						c_i \right) \right) \right)$ is $k$-calm, and therefore
						the polynomial $h'qh$ satisfies the condition of \pointref{it:altBecker}.
						
						To show that $q^*\left( h^*\left( c_i \right) \right)$ is $k$-precalm, we
						note that $h^*\left( c_i \right)$ is a sequence of polynomials, and thus
						all its $k$-calmnesses are nonnegative. By Lemma \ref{lem:actionEffect},
						the action of $q$ on this sequence will not change the $k$-calmnesses,
						because $q$ has nonzero valuation only at $k$-anxious roots of unity. Thus
						$q^*\left( h^*\left( c_i \right) \right)$ is precalm by Theorem
						\ref{thm:precalmness}. \qedhere
				\end{description}
			\end{proof}
			\noqedhere
		\end{proof}
	\section{Naive version of Becker's conjecture}\label{sect:naive}
		After looking at Theorems \ref{thm:orderOneCriteria} and
		\ref{thm:precalmness}, one might be tempted to propose the following naive
		version of Becker's conjecture.
		\begin{conjecture}[Becker, naive version]\label{conj:naive}
			Let $f\left( z \right) \in \ck\left[ \left[ z \right] \right]$ be a power
			series satisfying the Mahler equation
			\begin{equation*}
				f\left( z \right) = \sum_{i = 1}^{n} c_i\left( z \right) f\left(
				z^{k^i} \right).
			\end{equation*}
			Then $f$ is regular if and only if the sequence of coefficients
			$\left( c_i \right)$ is precalm.
		\end{conjecture}
		Were this true, it would provide us with a simple computational criterion for
		regularity of a Mahler series satisfying a given equation -- we could just
		compute the associated calmnesses and apply Theorem \ref{thm:precalmness}. Sadly, it is
		not true, and we will construct a power series satisfying a
		Mahler equation with a non-precalm sequence of coefficients, and which is
		nonetheless regular. Moreover, the Mahler equation for this series that we
		give is its Mahler equation of minimal degree.

		\begin{example}
			Let $k \geq 3$ and let $\alpha$ be a $k$-anxious number such that
			$\alpha^{k - 1} \neq \pm 1$ and $\alpha \neq \alpha^{k^i}\left( 1 +
			\alpha^{1 - k} \right)$ for all $i \in \N$.
			Consider the power series $f\left( z \right) \neq 0$
			satisfying the equation
			\begin{equation}\label{eq:example}
				f\left( z \right) = \frac{\left( \alpha^{k - 1} - \alpha^{1 - k} \right)z +
				\left( \alpha - \alpha^k \right)}{\alpha - z} f\left( z^k \right) +
				\frac{\alpha^k - z}{\alpha - z} f\left( z^{k^2} \right).
			\end{equation}
			Then $f$ is $k$-regular even though the sequence of coefficients in the equation is
			not $k$-precalm.
		\end{example}

		\begin{proof}
			We first show that a power series satisfying the given equation exists. We
			write $f\left( z \right) = \sum_{i = 0}^\infty a_i z^i$. Multiplying the
			equation \eqref{eq:example} by $\alpha - z$, we get the equation
			\begin{equation*}
				\left( \alpha - z \right) f\left( z \right) = \left( \left(
				\alpha^{k - 1} - \alpha^{1 - k} \right)z + \left( \alpha - \alpha^k \right)
				\right)f\left( z^k \right) + \left( \alpha^k - z \right)f\left( z^{k^2}
				\right).
			\end{equation*}
			Looking at the coefficients of degree zero in this equation we get
			\begin{equation*}
				\alpha a_0 = \left( \alpha - \alpha^k \right)a_0 + \alpha^k a_0 = \alpha a_0
			\end{equation*}
			which is satisfied for any $a_0 \in \ck$. Fix a nonzero term $0 \neq a_0 \in
			\ck$. Then the term $a_i$ is given by the recursive relation
			\begin{equation*}
				\alpha a_i - a_{i - 1} =
				\begin{cases}
					\left( \alpha - \alpha^k \right)a_q & i = kq, k \not\divides q,\\
					\left( \alpha^{k - 1} - \alpha^{1 - k} \right)a_q & i = kq + 1, k
					\not\divides q,\\
					\left( \alpha - \alpha^k \right)a_{kq} + \alpha^k a_q  & i = k^2q,\\
					\left( \alpha^{k - 1} - \alpha^{1 - k} \right)a_{kq} - a_q & i = k^2q +
					1,\\
					0 & \textrm{otherwise.}
				\end{cases}
			\end{equation*}
			This defines a nonzero sequence of coefficients of the formal power series $f$.
			
			We now show that the series $f\left( z \right)$ does not satisfy a Mahler equation of order lower
			than two.

			Suppose that $f$ satisfies an order one Mahler equation
			\begin{equation*}
				f\left( z \right) = c\left( z \right) f\left( z^k \right)
			\end{equation*}
			with $c\left( z \right) = \frac{p\left( z \right)}{q\left( z \right)}$,
			$p, q \in \ck\left[ z \right]$. Substituting $f\left( z^k \right) =
			\frac{1}{c\left( z \right)}f\left( z \right)$ and $f\left( z^{k^2} \right) =
			\frac{1}{c\left( z^k \right)c\left( z \right)}f\left( z \right)$ to the
			original equation, dividing it by $f\left( z \right)$ and multiplying by the
			denominators, we get
			\begin{multline*}
				\left( \alpha - z \right)p\left( z \right)p\left( z^k \right) = \\ \left(
				\left( \alpha^{k - 1} - \alpha^{1 - k} \right)z + \left( \alpha - \alpha^k
				\right)	\right) q\left( z \right) p\left( z^k \right) + \left( \alpha^k - z
				\right)q\left( z \right)q\left( z^k \right).
			\end{multline*}
			Set $P = \deg p\left( z \right)$, $Q = \deg q\left( z \right)$. Then the
			left hand side has degree $P + kP + 1$ and the right hand side at most degree
			$\max\left\{ Q + kP + 1, Q + kQ + 1 \right\}$ and the degree is in fact equal to
			$\max\left\{ Q + kP + 1, Q + kQ + 1 \right\}$ if $P \neq Q$. If $P > Q$, then
			we get a contradiction,
			since then $P + kP + 1 > Q + kP + 1 > Q + kQ + 1$. If $Q > P$, then we also
			have a contradiction: $Q + kQ + 1 > Q + kP + 1 > P + kP + 1$. Therefore
			$Q = P$. Since $p\left( z^k \right)$ divides two of the three terms, and is coprime to
			$q\left( z^k \right)$, we have $p\left( z^k \right) \divides \left( \alpha^k - z
			\right)q\left( z \right)$. This gives us the inequality $kP \leq Q + 1 = P
			+ 1$, but this is a contradiction if $P \neq 0$ since $k > 2$. If $P = Q = 0$, then
			Adamczewski and Bell \cite[Lemma 7.1]{AdamczewskiBell13} have shown that $f$ has to be
			constant, which is not the case. The argument is even simpler in this case, so we recall
			it. Let $f\left( z \right)$ satisfy the Mahler equation
			\begin{equation*}
				f\left( z \right) = c f\left( z^k \right)
			\end{equation*}
			for some constant $c \in \ck$. If $f\left( z \right) = \sum_{i =
			1}^{\infty} a_i z^i$ is not constant, then
			there exists the lowest integer $1 \leq i$ such that $a_i \neq 0$. However,
			on the right side of the equation the coefficient of $z^i$ is $0$, which is a
			contradiction.
			By the above argument $f$ satisfies no Mahler equations
			of order less than two.

			Let $c_1\left( z \right) = \frac{\left( \alpha^{k - 1} - \alpha^{1 - k}
			\right)z + \left( \alpha - \alpha^k \right)}{\alpha - z}$ and $c_2\left( z
			\right) = \frac{\alpha^k - z}{\alpha - z}$. For any sequence $\left( a_i
			\right) \in A_{\alpha}$ with $a_0 = a_1 = 1$ the calmness $\clm_{\alpha,
			\left( a_i \right)}\left( c_i \right) = -1$ is negative. By Theorem
			\ref{thm:precalmness}, this sequence cannot be precalm.

			Consider the operator $1 - c_1\left( z \right)\Delta_k - c_2\left( z
			\right)\Delta_k^2 \in \ck\left( z \right)\left[ \Delta_k \right]$
			corresponding to the equation satisfied by $f$. This operator annihilates $f$
			and hence so does the operator obtained from it after multiplying it on the left by
			$1 + \psi\left( z \right)\Delta_k$ with $\psi\left( z \right) = -\alpha^{1 -
			k} \frac{\alpha - z^k}{\alpha - z}$. We have
			\begin{multline*}
				\left( 1 + \psi\left( z \right)\Delta_k \right)\left( 1 - c_1\left( z
				\right)\Delta_k - c_2\left( z \right)\Delta_k^2 \right) = \\ 1 - \left(
				c_1\left( z \right) - \psi\left( z \right) \right)\Delta_k - \left(
				c_2\left( z \right) + \psi\left( z \right)c_1\left( z^k \right)
				\right)\Delta_k^2 - \psi\left( z \right)c_2\left( z^k \right)\Delta_k^3
			\end{multline*}
			This corresponds to a Mahler equation of order three satisfied by $f$. The
			coefficients of this equation are polynomials. The coefficient at $\Delta_k$
			is
			\begin{equation*}
				-\left( c_1\left( z \right) - \psi\left( z \right) \right) = -\frac{\left(
				\alpha^{k - 1} - \alpha^{1 - k} \right)z + \left( \alpha - \alpha^k \right)
				+ \alpha^{1 - k}\left( \alpha - z^k \right)}{\alpha - z}
			\end{equation*}
			and it is a polynomial, because the numerator has a zero at $\alpha$. The
			coefficient at $\Delta_k^2$ is
			\begin{equation*}
				\begin{split}
					-\left( c_2\left( z \right) + \psi\left( z \right)c_1\left( z^k \right)
					\right) &=
					-\frac{\alpha^k - z}{\alpha - z} + \frac{\alpha^{1 - k}\left( \alpha - z^k
					\right) \left( \left( \alpha^{k - 1} - \alpha^{1 - k} \right)z^k + \left(
					\alpha - \alpha^k \right) \right)}{\left( \alpha - z \right)\left( \alpha -
					z^k \right)} \\
					&= -\frac{\left( \alpha^{2 - 2k} - 1 \right)z^k - z + \alpha^k -
					\alpha^{2 - k} + \alpha}{\alpha - z}
				\end{split}
			\end{equation*}
			and is a polynomial for the same reason. So is the coefficient at $\Delta_k^3$
			\begin{equation*}
				-\psi\left( z \right)c_2\left( z^k \right) = \frac{\alpha^{1 - k}\left(
				\alpha - z^k \right)\left( \alpha^k - z^k \right)}{\left( \alpha - z
				\right)\left( \alpha - z^k \right)} = \frac{\alpha^{1 - k}\left( \alpha^k -
				z^k \right)}{\alpha - z}.
			\end{equation*}
			By Theorem \ref{thm:DumasReguralityCriterion}, this implies that $f$ is
			regular.
			\noqedhere
		\end{proof}
	\section{The criterion for regularity}\label{sect:criteria}
	 In this final section we prove the following theorem.
		\begin{theorem}\label{thm:fullEquivalence}
			Let $f\left( z \right) \in \ck\left[ \left[ z \right] \right]$ be a Mahler power
			series. Assume that the coefficients of the minimal order Mahler equation satisfied
			by $f$ have no poles at roots of unity. Then $f$ is
			regular if and only if it satisfies a Mahler equation
			\begin{equation*}
				f\left( z \right) = \sum_{i = 1}^{n} c_i\left( z \right) f\left(
				z^{k^i} \right)
			\end{equation*}
			with $\left( c_i \right)_{i = 1}^n$ a calm sequence of rational functions.
		\end{theorem}
		We also conjecture that the assumptions on coefficients of the minimal
		equation are unneccessary.
		\begin{conjecture}\label{conj:fullEquivalence}
			A power series $f\left( z \right) \in \ck\left[ \left[ z \right] \right]$ is
			regular if and only if it satisfies a Mahler equation
			\begin{equation*}
				f\left( z \right) = \sum_{i = 1}^{n} c_i\left( z \right) f\left(
				z^{k^i} \right)
			\end{equation*}
			with $\left( c_i \right)_{i = 1}^n$ a calm sequence of rational functions.
		\end{conjecture}
		Conjecture \ref{conj:fullEquivalence} trivially implies Conjecture
		\ref{conj:altBecker} and therefore	it also implies the
		original Becker's conjecture (see the discussion of various versions of Becker's
		conjecture in Section \ref{sect:conjectures}). By the same arguments the
		following weaker version of Becker's conjecture is also implied by Theorem
		\ref{thm:fullEquivalence}.
		\begin{proposition}\label{prop:weakBecker}
			Let $f\left( z \right) \in \ck\left[ \left[ z \right] \right]$ be a Mahler power
			series. Assume that the coefficients of the minimal order Mahler equation satisfied
			by $f$ have no poles at roots of unity. Then $f$ is $k$-regular
			if and only if there exists a $k$-regular rational function $0 \neq h\left( z \right)
			\in \ck\left( z \right)$ such that $\frac{f\left(
			z \right)}{h\left( z \right)}$ satisfies a Mahler equation with polynomial
			coefficients.
		\end{proposition}

		The statement of Theorem \ref{thm:fullEquivalence} is purely
		existential. However, in the process of proving the theorem we will explicitly
		give a computational criterion involving coefficients of the minimal order
		Mahler equation satisfied by $f$.
		\subsection{Generating higher order Mahler equations}
			Consider a power series $f\left( z \right) \in \ck\left[ \left[ z
			\right] \right]$ satisfying the Mahler equation
			\begin{equation*}
				f\left( z \right) = \sum_{i = 1}^{n} c_i\left( z \right) f\left(
				z^{k^i} \right).
			\end{equation*}
			This equation corresponds to the operator
			\begin{equation*}
				1 - \sum_{i = 1}^n c_i\left( z \right)\Delta_k^i
			\end{equation*}
			that annihilates $f$. Consider a sequence of rational functions $\left(
			\psi_{i}\left( z \right) \right)$ and consider the following operators
			\begin{equation*}
				\Gamma_m\left( \psi_i \right) = \left( 1 + \sum_{i = 1}^m \psi_i\left( z \right)\Delta_k^i \right)\left(
				1 - \sum_{i = 1}^n c_i\left( z \right)\Delta_k^i \right) = \left( 1 -
				\sum_{i = 1}^{n + m} d_{i, m}\left( z \right)\Delta_k^i \right).
			\end{equation*}
			These operators also annihilate $f$ and we have the following recursive formulas for
			$d_{i, m}$:
			\begin{align*}
				d_{i, 0}\left( z \right) &= c_i\left( z \right),\\
				d_{i, m}\left( z \right) &=
				\begin{cases}
					d_{i, m - 1}\left( z \right) & \textrm{ for } 0 \leq i < m\\
					d_{i, m - 1}\left( z \right) - \psi_m\left( z \right) & \textrm{ for } i = m\\
					d_{i, m - 1}\left( z \right) + \psi_m\left( z \right)c_{i - m}\left( z^{k^m}
					\right) & \textrm{ for } m < i < m + n\\
					\psi_m\left( z \right)c_{i - m}\left( z^{k^m}	\right) & \textrm{ for } i =
					m + n.
				\end{cases}
			\end{align*}
			These formulas can be proven by a simple induction. The case $m = 0$ is trivial.
			Assume the claim is true for $m - 1$. Then we get
			\begin{multline*}
					\left( 1 + \sum_{i = 1}^m \psi_i\left( z \right)\Delta_k^i \right)\left(
					1 - \sum_{i = 1}^n c_i\left( z \right)\Delta_k^i \right) =\\
					\left( 1 + \sum_{i = 1}^{m - 1} \psi_i\left( z \right)\Delta_k^i \right)\left( 1 -
					\sum_{i = 1}^n c_i\left( z \right)\Delta_k^i \right)
					+ \psi_m\left( z
					\right)\Delta_k^m \left( 1 - \sum_{i = 1}^n c_i\left( z \right)\Delta_k^i
					\right) =\\
					\left( 1 - \sum_{i = 1}^{n + m - 1} d_{i, m - 1}\left( z
					\right)\Delta_k^i \right) + \left( \psi_m\left( z \right)\Delta_k^m -
					\sum_{i = 1}^n \psi_m\left( z \right)c_i\left( z^{k^m}
					\right)\Delta_k^{m + i} \right) = \\
					1 - \left(\sum_{i = 1}^{m - 1}
					d_{i, m - 1}\left( z \right)\Delta_k^i \right)
					- \left( d_{m, m - 1}\left(
					z \right) - \psi_m\left( z \right) \right)\\
				 - \left(\sum_{i = m + 1}^{n + m - 1}
					d_{i, m - 1}\left( z \right) + \psi_m\left( z \right)c_{i - m}\left(
					z^{k^m} \right) \right)
					- \psi_m\left( z \right)c_{n}\left( z^{k^m}
					\right).
			\end{multline*}

			We now define a particularly interesting family of operators of the above
			type.
			\begin{definition}
				Let the sequence of rational functions $\left( \psi_i\left( z \right)
				\right)$ be defined recursively by the formula
				\begin{equation*}
					\psi_m\left( z \right) = d_{m, m - 1}\left( z \right).
				\end{equation*}
				This definition is not circular, since the value of $d_{m, m-1}\left( z \right)$ depends only on
				$\psi_i\left( z \right)$ for $1 \leq i < m$. We obtain in this manner a
				sequence of operators $\Gamma_m$ annihilating $f$. We call these operators
				the \emph{special operators for $f$}.
			\end{definition}
			It is easy to see from the recursive formulas for $d_{i, m}\left( z
			\right)$ that the special operators of $f$ have the form
			\begin{equation*}
				\Gamma_m = 1 - \sum_{i = 1}^n d_{m+i, m}\left( z \right)\Delta_k^{m + i}
			\end{equation*}
			and correspond to Mahler equations satisfied by $f$ of the form
			\begin{equation*}
				f\left( z \right) = \sum_{i = 1}^n d_{m+i, m}\left( z \right)f\left(
				z^{k^{i + m}} \right).
			\end{equation*}

			Defining the special operators for $f$ allows us to state a more precise
			version of Theorem \ref{thm:fullEquivalence}.
			\begin{proposition}\label{prop:technicalFullEquivalence}
				Let $f\left( z \right) \in \ck\left[ \left[ z \right] \right]$ be a power
				series satisfying the Mahler equation
				\begin{equation*}
					f\left( z \right) = \sum_{i = 1}^{n} c_i\left( z \right) f\left(
					z^{k^i} \right)
				\end{equation*}
				and satisfying no Mahler equations of order lower than $n$. Define the set
				\begin{equation*}
					\cV =
					\left\{ \alpha \in \ck \colon \alpha \textrm{ is anxious and for some } 1 \leq i \leq
					n, v_\alpha\left( c_i\left( z \right)	\right) < 0 \right\}.
				\end{equation*}
			 If $\cV$
				contains no roots of unity, then the following conditions are equivalent
				\begin{enumerate}[(i)]
					\item\label{it:technicalNonroot} For every $\alpha \in \cV$ there exists
						a natural number $m$ such that $v_\alpha\left( d_{i, m}\left( z \right) \right) \geq 0$
						for all $m + 1 \leq i \leq m + n$ and $v_{\alpha^{k^j}}\left( c_i\left( z \right) \right) \geq 0$
						for all $j > m$, $1 \leq i \leq n$ (note that the latter is automatically
						satisfied for $m$ large enough).
					\item\label{it:hasCalmEquation} $f$ satisfies a Mahler equation with a calm sequence of coefficients.
					\item\label{it:isRegular} $f$ is regular.
				\end{enumerate}
				In condition \pointref{it:technicalNonroot}, $d_{i, m}\left( z \right)$ are
				the coefficients of the special operators for $f$.
			\end{proposition}
		 Theorem	\ref{thm:fullEquivalence} states the equivalence
			of conditions \pointref{it:hasCalmEquation} and \pointref{it:isRegular} in
			this proposition.

			The criterion stated in Proposition \ref{prop:technicalFullEquivalence} is
			effective and, given a Mahler equation for $f$ of minimal order, allows to
			check whether $f$ is regular. In fact, rational functions have only finitely
			many poles and roots, so at some point they stop influencing the valuations of
			$d_{i, m}\left( z \right)$.
		\subsection{Proof of the criterion}
			\begin{proof}
				The implication
				\pointref{it:hasCalmEquation}$\Rightarrow$\pointref{it:isRegular} follows
				immiediately from Theorem \ref{thm:DumasReguralityCriterion}.
				\begin{description}
					\item[\pointref{it:technicalNonroot}$\Rightarrow$\pointref{it:hasCalmEquation}]
						First note that if condition \pointref{it:technicalNonroot} is
						satisfied for a number $m$, then it is also satisfied for any number
						bigger than $m$. Indeed, since $v_{\alpha^{k^j}}\left( c_i\left( z \right)
						\right) \geq 0$ for all $j > m$, $1 \leq i \leq n$, the functions
						$c_i\left( z \right)$ will not lower any valuations of $d_{i, m'}\left(
						z \right)$ for $m' \geq m$ and sums of
						terms with nonnegative valuation cannot have negative valuation.
						
						Since $\cV$ is a finite set, we can pick a number $m_0$
						such that	property \pointref{it:technicalNonroot} is satisfied for all
						$\alpha \in \cV$. We therefore have that $d_{i, m_0}\left( z \right)$ have
						no poles at $\alpha \in \cV$. We will construct a sequence of rational
						functions $\left( \psi'_l \right)_{l = 1}^{m_0}$ such that $\Gamma_m\left(
						\psi'_l \right) = 1 - \sum_{i = 1}^{n
						+ m} d'_{i, m}\left( z \right)\Delta_k^i$ for $1 \leq m \leq m_0$ and such
						that the coefficients $d'_{i, m_0}\left( z \right)$ are calm.

						The choice of $\psi'_1, \dots, \psi'_{m - 1}$ determines $d'_{i, m - 1}$
						and these values will be used in turn to construct $\psi'_m$. Note that
						$d'_{i, 0}\left( z \right) = c_i\left( z \right)$ does not depend on the
						choice of $\psi'_l$. Consider the number 
						\begin{equation*}
							v = -\min_{\substack{\alpha \in \cV \\ 0 \leq m \leq m_0\\ 1+m \leq j \leq n+m}}\left\{ v_\alpha\left(
							d_{j, m}\left( z \right) \right)\right\}.
						\end{equation*}
						Note that $v > 0$ since $\alpha \in \cV$.
						Define the set
						\begin{equation*}
							B_m = \left\{ \alpha \in \ck
							\colon v_{\alpha^{k^m}}\left( c_j\left( z \right) \right) < 0
							\textrm{ for some } 1 \leq j \leq n \right\} \setminus \cV.
						\end{equation*}
						Set
						\begin{equation*}
							\psi_m'\left( z \right) =
							h\left( z \right) d'_{m, m - 1}\left( z \right) \prod_{\alpha \in B_m}
							\left( \alpha - z \right)^{-\min_{1 \leq j \leq n}\left\{ v_{\alpha^{k^m}}\left( c_j\left(
							z \right) \right) \right\}},
						\end{equation*}
						where $h\left( z \right) \in \ck\left[ z \right]$ is chosen by Lemma \ref{lem:arbitraryDigits} so that $\psi'_m$ has the
						same $v$ first $\left( z - \alpha \right)$-adic digits as $d_{m, m - 1}\left( z
						\right)$ for any $\alpha \in \cV$. Such a choice of $h$ is possible since,
						as we will show shortly, $d_{m, m - 1}\left( z \right)$ has the same
						$\left( z - \alpha \right)$-adic valuation as $d'_{m, m - 1}\left( z
						\right)$ for $\alpha \in \cV$.

						We first show that $d'_{i, m_0}\left( z \right)$ are $\alpha$-calm for $\alpha \in \cV$. It
						is sufficient to show that $d'_{i, m}\left( z \right)$ is $\alpha$-calm
						for $1 \leq i \leq m$ and has the same $v$ first
						$\left( z - \alpha \right)$-adic digits as $d_{i, m}\left( z \right)$ for
						all $m + 1 \leq i \leq m+n$, $0 \leq m \leq m_0$. (Here we say that the
						power series have the same first $s$ digits if their
						valuations have the same value (say $w$) and the digits at $\left( z -
						\alpha	\right)^j$ are the same for $w \leq j \leq w + s - 1$.) Indeed, in
						this case we know that $d'_{i, m}$ and $d_{i, m}$ have the same $\left( z - \alpha
						\right)$-adic valuation and that $d_{i, m_0}\left( z \right)$ has positive
						valuation at all $\alpha \in \cV$ by condition \pointref{it:technicalNonroot}.
						We prove this statement inductively on $m$. The functions $d'_{i, 0}\left( z
						\right) = c_i\left( z \right) = d_{i, 0}\left( z \right)$, so the condition
						is satisfied. Assume that $d'_{i, m - 1}\left( z \right)$ is $\alpha$-calm
						for all $1 \leq i \leq m - 1$ and has the same
						first $\left( z - \alpha \right)$-adic digit as $d_{i, m - 1}\left( z
						\right)$ for $m \leq i \leq m+n-1$. Recall the formulas for the functions
						\begin{align*}
							d'_{i, m}\left( z \right) &=
							\begin{cases}
								d'_{i, m - 1}\left( z \right) & \textrm{ for } 0 \leq i < m\\
								d'_{i, m - 1}\left( z \right) - \psi'_m\left( z \right) & \textrm{ for } i = m\\
								d'_{i, m - 1}\left( z \right) + \psi'_m\left( z \right)c_{i - m}\left( z^{k^m}
								\right) & \textrm{ for } m < i < m + n\\
								\psi'_m\left( z \right)c_{i - m}\left( z^{k^m}	\right) & \textrm{ for } i =
								m + n.
							\end{cases}
						\end{align*}
						For $0 \leq i \leq m - 1$, $d'_{i, m}\left( z \right)$ is
						$\alpha$-calm since $d'_{i, m - 1}\left( z \right)$ is $\alpha$-calm for
						all $\alpha \in \cV$. For $i = m$ note that $d'_{i, m - 1}\left( z
						\right)$ has the same $v$ first $\left( z - \alpha \right)$-adic digits as
						$d_{i, m - 1}\left( z \right)$, so also as $\psi'_m\left( z \right)$ by
						construction. Therefore, by the choice of $v$, $d'_{i, m} = d'_{i, m -
						1}\left( z \right) - \psi'_m\left( z \right)$ has positive $\left( z -
						\alpha \right)$-adic valuation, so is $\alpha$-calm. For $m + 1 \leq i
						\leq m + n$ again $d'_{i, m - 1}\left( z \right)$ has the same $v$ first
						$\left( z - \alpha \right)$-adic digits as $d_{i, m - 1}\left( z \right)$
						and $\psi'_m\left( z \right)$ has the same $v$ first $\left( z - \alpha
						\right)$-adic digits as $d_{m, m - 1}\left( z \right) = \psi_m\left( z
						\right)$. Therefore all the terms in the formula for $d'_{i, m}\left( z
						\right)$ have the same $v$ first $\left( z - \alpha \right)$-adic digits
						as the terms in the formula for $d_{i, m}\left( z \right)$, so these
						functions also have the same $v$ first $\left( z - \alpha \right)$-adic
						digits finishing the proof by induction.

						Now we show that $d'_{i, m}\left( z \right)$ are $\alpha$-calm for all
						$\alpha \not\in \cV$ and $0 \leq m \leq m_0$. We again show this inductively.
						For $m = 0$ the function $d'_{i, 0}\left( z \right) = c_i\left( z
						\right)$ is $\alpha$-calm for $\alpha \not\in \cV$ by the definition of
						$\cV$. Assume that $d'_{i, m - 1}\left( z \right)$ are $\alpha$-calm for
						$\alpha \not\in \cV$. Note that then $\psi'_m\left( z \right)$ is also
						$\alpha$-calm. Indeed, 
						\begin{equation*}
							\psi_m'\left( z \right) =
							h\left( z \right) d'_{m, m - 1}\left( z \right) \prod_{\alpha \in B_m}
							\left( \alpha - z \right)^{-\min_{1 \leq j \leq n}\left\{ v_{\alpha^{k^m}}\left( c_j\left(
							z \right) \right) \right\}},
						\end{equation*}
						where $h\left( z \right)$ and the product $\prod_{\alpha
						\in B_m} \left( \alpha - z \right)^{-\min_{1 \leq j \leq n}
						v_{\alpha^{k^m}}\left( c_j\left( z \right) \right)}$ are
						polynomials and $d'_{m, m - 1}\left( z \right)$ is $\alpha$-calm.
						Therefore, $d'_{m, m}\left( z \right)$ is $\alpha$-calm
						as a sum of two $\alpha$-calm functions. The function $d'_{i, m}\left(
						z \right)$, $m + 1 \leq i \leq m + n$ is also a sum of $\alpha$-calm functions. Indeed, if $\alpha
						\not\in B_m$, then $c_{i - m}\left( z^{k^m} \right)$ is $\alpha$-calm and
						if $\alpha \in B_m$, then the valuation
						\begin{align*}
							&v_{\alpha}\left( \psi'_m\left( z \right) c_{i - m}\left( z^{k^m} \right)
							\right) =\\
							&v_{\alpha}\left( h\left( z \right) \right) + v_{\alpha}\left(
							\prod_{\alpha \in B_m} \left( \alpha - z \right)^{-\min_{1 \leq j \leq
							n}v_{\alpha^{k^m}}\left( c_j\left( z \right) \right)} \right) +\\
							&v_{\alpha}\left( d'_{m, m - 1}\left( z
							\right) \right) + v_{\alpha}\left( c_{i - m}\left( z^{k^m} \right)
							\right) \geq\\
							&-\min_{1 \leq j \leq n}v_{\alpha^{k^m}}\left( c_j\left(
							z \right) \right) + v_{\alpha^{k^m}}\left( c_{i - m}\left( z \right)
							\right) \geq 0.
						\end{align*}
						This finishes the proof by induction.

						Therefore, we constructed an operator corresponding to a Mahler equation
						of order $n + m_0$ satisfied by $f\left( z \right)$ with a calm sequence
						of coefficients $\left( d'_{i, m_0} \right)_{i = 1}^{n + m_0}$.
					\item[\pointref{it:isRegular}$\Rightarrow$\pointref{it:technicalNonroot}]
						As in Theorem \ref{thm:orderOneCriteria}, we follow the idea of the proof
						in Dumas \cite[Proposition 54]{Dumas93}. Consider the
						equations corresponding to the special operators for $f$. We can apply $m + 1$
						arbitrary Cartier operators to the equation $\Gamma_m f = 0$, which gives us
						\begin{equation}\label{eq:lambdaOnSpecial}
							\Lambda_{r_m}\left( \dots \left( \Lambda_{r_0}\left( f\left( z \right) \right)
							\right) \right) = \sum_{i = 1}^n \Lambda_{r_m}\left( \dots \left( \Lambda_{r_0}\left(
							d_{m + i, m}\left( z \right)	\right)\right)\right)f\left( z^{k^{i - 1}} \right).
						\end{equation}
						However, by assumption that $f$ satisfies no Mahler equations of order less
						than $n$, the terms $f\left( z^{k^{i - 1}} \right)$ on the right hand side are
						$\ck\left( z \right)$-linearly independent. If the $\ck$-vector space $V$
						spanned by $\Lambda_{r_m}\left( \dots \left( \Lambda_{r_0}\left( f\left( z \right)
						\right) \right) \right)$ was finitely dimensional, then its generators would
						be of the form 
						\begin{equation*}
							\sum_{i = 1}^{n} h_i\left( z \right)f\left(
							z^{k^{i - 1}}	\right)
						\end{equation*}
							for some rational functions $h_i\left( z \right) \in \ck\left( z
							\right)$. Since $h_i$ have only finitely many
						poles of finite order and by the $\ck\left( z \right)$-linear
						independence of $f\left( z^{k^{i - 1}} \right)$, $V \subset h\left( z
						\right)\sum_{i = 1}^{n}\ck\left[ z \right]f\left( z^{k^{i - 1}} \right)$
						for some $0 \neq h \in \ck \left( z \right)$. It remains to prove that if
						condition \pointref{it:technicalNonroot} is not satisfied, then we can find
						elements of $V$ for which this is not true.

						If condition \pointref{it:technicalNonroot} is not satisfied, then we can
						pick arbitrarily large $m$ such that a certain $d_{i, m}\left( z \right)$
						has a pole at some $\alpha \in \cV$. Then, by a repeated application of Lemma
						\ref{lem:polePerserving}, we can choose $r_i$ ($0 \leq i \leq m$) so that
						\begin{equation*}
							0 > v_\alpha\left( \Lambda_{r_m}\left( \dots \Lambda_{r_0}\left( d_{i, m}
							\right) \right) \left( z^{k^{m+1}} \right) \right).
						\end{equation*}
						However,
						\begin{equation*}
							v_\alpha\left( \Lambda_{r_m}\left( \dots \Lambda_{r_0}\left( d_{i, m}
							\right) \right) \left( z^{k^{m+1}} \right) \right) =
							v_{\alpha^{k^{m + 1}}}\left( \Lambda_{r_m}\left( \dots \Lambda_{r_0}\left(
							d_{i, m} \right) \right) \right).
						\end{equation*}
						Therefore, we see by \eqref{eq:lambdaOnSpecial} that $h\left( z
						\right)$ has a pole at $\alpha^{k^{m + 1}}$ for arbitrarily large $m$.
						Since $\alpha$ is not a root of unity this shows that $h\left( z
						\right)$ has infinitely many poles which yields a contradiction, since $h$
						is a rational function. \qedhere
				\end{description}
			\end{proof}

	Tomasz Kisielewski,\\Institute of Mathematics,\\Jagiellonian University,\\ul. prof. Stanisława Łojasiewicza 6,\\30-348 Kraków,\\Poland,\\e-mail: \url{tomasz.kisielewski@uj.edu.pl}
\end{document}